\documentclass[11pt]{article}

\usepackage{amsmath}
\usepackage{amssymb}
\usepackage{amsfonts}
\usepackage{amsthm}
\usepackage{array}
\usepackage{bm}
\usepackage{enumitem}

\newcolumntype{C}[1]{>{\centering\hspace{0pt}}p{#1}}
\usepackage[pdftex]{graphicx}
\usepackage[margin=1in]{geometry}

\newcommand{\R}{\mathbb{R}}

\newcommand{\Hc}{\mathcal{H}}
\newcommand{\Vc}{\mathcal{V}}

\newtheorem{thm}{Theorem}[section]
\newtheorem{prop}[thm]{Proposition}
\newtheorem{lem}[thm]{Lemma}
\newtheorem{cor}[thm]{Corollary}

\theoremstyle{definition}
\newtheorem{defn}[thm]{Definition}

\theoremstyle{definition}
\newtheorem*{rmk}{Remark}

\usepackage[nottoc]{tocbibind}
\numberwithin{equation}{section}

\usepackage[hidelinks]{hyperref}

\title{The Mean Curvature of Special Lagrangian 3-folds in $\text{SU}(3)$-Structures with Torsion}
\author{Gavin Ball and Jesse Madnick}
\date{December 2020}

\newcommand{\Addresses}
{{  \bigskip
  \textsc{Universit\'{e} du Qu\'{e}bec \`{a} Montr\'{e}al} \par\nopagebreak
  \textsc{D\'{e}partement de math\'{e}matiques}\par\nopagebreak
  \textsc{Case postale 8888, succursale centre-ville}\par\nopagebreak
    \textsc{Montr\'{e}al (Qu\'{e}bec), H3C 3P8, Canada}\par\nopagebreak
 \textit{E-mail address}: \texttt{gavin.cf.ball@gmail.com} \\

 \bigskip
  \textsc{National Center for Theoretical Sciences} \par\nopagebreak
  \textsc{National Taiwan University}\par\nopagebreak
  \textsc{Taipei, Taiwan}\par\nopagebreak
 \textit{E-mail address}: \texttt{jmadnick@ncts.ntu.edu.tw}

}}

\begin{document}

\maketitle

\begin{abstract}
We derive formulas for the mean curvature of special Lagrangian $3$-folds in the general case where the ambient 6-manifold has intrinsic torsion.  Consequently, we are able to characterize those $\text{SU}(3)$-structures for which every special Lagrangian $3$-fold is a minimal submanifold. In the process, we obtain an obstruction to the local existence of special Lagrangian $3$-folds.
\end{abstract}

\tableofcontents
 
\section{Introduction}

\indent \indent Let $(M^6, g)$ be a Riemannian 6-manifold.  If the holonomy group of $g$ lies in $\text{SU}(3)$, then on $M$ there exist a parallel orthogonal complex structure $J$, a parallel $(1,1)$-form $\Omega \in \Omega^{1,1}(M)$ given by $\Omega = g(J \cdot, \cdot)$, and a non-vanishing parallel $(3,0)$-form $\Upsilon \in \Omega^{3,0}(M)$ normalized to have $\Upsilon \wedge \overline{\Upsilon} = -\frac{4}{3}i\,\Omega^3$.  The manifold $M$ with the parallel tensors $(g, \Omega, J, \Upsilon)$ is called a \textit{Calabi-Yau $3$-fold}. \\
\indent A \textit{special Lagrangian (SL) 3-fold of phase $\theta \in [0, 2\pi)$} is a 3-dimensional submanifold $L \subset M$ satisfying
%$$\text{Re}( e^{-i\theta}\Upsilon )|_L = \pm 1,$$
%or equivalently,
$$\Omega|_L = 0 \ \ \ \ \text{ and } \ \ \ \ \text{Im}(e^{-i\theta}\Upsilon)|_L = 0.$$
Special Lagrangian $3$-folds were first defined by Harvey and Lawson in their foundational 1982 paper on calibrations \cite{harvey1982calibrated}.  In that work, they proved that SL 3-folds in Calabi-Yau 3-folds exist locally in abundance, depending on $2$ arbitrary functions of $2$ real variables, in the sense of exterior differential systems.  Moreover, they showed that $\text{Re}( e^{-i\theta}\Upsilon )$ is a calibration whose calibrated submanifolds are exactly the SL $3$-folds of phase $\theta$, and hence every SL 3-fold is homologically volume-minimizing.  In particular, SL 3-folds are minimal submanifolds. \\
\indent Since 1982, SL 3-folds in Calabi-Yau $3$-folds have been the subject of intense study.  For example, by imposing various geometric assumptions, large families of explicit SL $3$-folds are now known (and, in many cases, classified): see, for example, \cite{bryant2000second, bryant2006so, joyce20010, joyce2002ruled, joyce2002special, ionel2008cohomogeneity}.  Singular SL $3$-folds have also been constructed, such as cones with interesting topologies \cite{haskins2007special} and pairs of intersecting planes \cite{lawlor1989angle}.  In turn, gluing methods and desingularization schemes based on these singular models have proven highly successful in constructions: see, for example, \cite{butscher2001regularizing, lee2003connected, lee2003embedded, pacini2013special}. The deformation theory of closed special Lagrangians was first studied by McLean in his 1990 PhD thesis \cite{mclean1998deformations}, and furthered by Hitchin \cite{hitchin1997moduli}. Special Lagrangians are also of interest in physics by way of the SYZ Conjecture \cite{strominger1996mirror}, which explains mirror symmetry in terms of fibrations of Calabi-Yau $3$-folds by SL $3$-tori. \\
 % The compact with boundary case is studied by Butscher \cite{butscher2003deformations}, and the asymptotically conical case independently by Marshall \cite{marshal2002deformations} and Pacini \cite{pacini2004deformations}.

\indent Special Lagrangian $3$-folds may also be studied in a class of ambient spaces more general than Calabi-Yau $3$-folds: namely, those Riemannian $6$-manifolds $(M^6, g)$ with an $\text{SU}(3)$-structure $(J, \Omega, \Upsilon)$ that is merely $g$-compatible.  In other words, we drop the requirement that the holonomy of $g$ lie in $\text{SU}(3)$, so that the tensors $(J, \Omega, \Upsilon)$ above need no longer be parallel. \\
\indent The generalization from $\text{SU}(3)$-holonomy to $\text{SU}(3)$-structures leads to several interesting geometries.  Indeed, many classes of $\text{SU}(3)$-structures --- such as nearly-K\"{a}hler, half-flat, and balanced --- enjoy relationships with special holonomy and physics.  For example, nearly-K\"{a}hler $\text{SU}(3)$-structures admit real Killing spinors, and hence their (appropriately scaled) Riemannian cones have $\text{G}_2$-holonomy and parallel spinors \cite{grunewald1990six, bar1993real}.  Half-flat $\text{SU}(3)$-structures yield $\text{G}_2$-manifolds via evolution equations \cite{hitchin2001stable}, and balanced $\text{SU}(3)$-structures arise in string theory via the Strominger system \cite{strominger1986superstrings, li2005existence}.  Deformations of special Lagrangian 3-folds in $6$-manifolds with torsion are studied in \cite{Salur00}. \\
\indent However, in this generalized setting, the form $\text{Re}(e^{-i\theta}\Upsilon)$ need not be a calibration, and thus SL $3$-folds of phase $\theta$ need not be volume-minimizing.  Despite this, the particular nature of the $\text{SU}(3)$-structure may nevertheless force SL $3$-folds to be minimal.  For example, when the $\text{SU}(3)$-structure on $M$ is nearly-K\"{a}hler, every SL $3$-fold $L \subset M$ is minimal as its Riemannian cone $C(L) \subset C(M)$ is homologically volume-minimizing (as $C(L)$ is a coassociative $4$-fold).  This raises the natural: \\

\noindent \textbf{Minimality Problem:} Let $M$ be a $6$-manifold.  Characterize those $\text{SU}(3)$-structures on $M$ for which every special Lagrangian $3$-fold in $M$ is a minimal submanifold of $M$. \\

\indent In this paper, we will solve the Minimality Problem by deriving a simple formula for the mean curvature of special Lagrangian $3$-folds.  Analogous formulas for submanifolds arising in $\text{G}_2$ geometry and $\text{Spin}(7)$ geometry are derived in our companion paper \cite{BaMaExcept}. \\

\indent Perhaps more fundamentally, in our generalized context special Lagrangian $3$-folds of a given phase need not exist at all, even locally.  This raises the natural: \\

\noindent \textbf{Local Existence Problem:} Let $M$ be a $6$-manifold.  Characterize those $\text{SU}(3)$-structures on $M$ for which special Lagrangian $3$-folds exist locally at every point of $M$. \\

\indent In this work, we make progress towards the resolution of the Local Existence Problem.  More precisely, we obtain new obstructions to the local existence of special Lagrangian $3$-folds.  Analogous obstructions to the local existence of coassociative 4-folds are obtained in \cite{BaMaExcept}.

\subsection{Main Results}

\indent \indent Let $(M^6, \Omega, \Upsilon)$ be a $6$-manifold with an $\text{SU}(3)$-structure $(\Omega, \Upsilon) \in \Omega^2(M) \oplus \Omega^3(M; \mathbb{C})$, and let $J$ be the underlying almost-complex structure.  The first-order local invariants of $(\Omega, \Upsilon)$ are completely encoded in seven differential forms, called the \textit{torsion forms} of the $\text{SU}(3)$-structure, denoted
$$(\tau_0, \widehat{\tau}_0, \tau_2, \widehat{\tau}_2, \tau_3, \tau_4, \tau_5) \in \Omega^0 \oplus \Omega^0 \oplus \Omega^2 \oplus \Omega^2 \oplus \Omega^3 \oplus \Omega^1 \oplus \Omega^1$$
They are defined by the equations
\begin{align*}
d\Omega & = 3\tau_0\,\text{Re}(\Upsilon) + 3\widehat{\tau}_0\,\text{Im}(\Upsilon) + \tau_3 + \tau_4 \wedge \Omega \\
d\,\text{Re}(\Upsilon) & = 2\widehat{\tau}_0\,\Omega^2 + \tau_5 \wedge \text{Re}(\Upsilon) + \tau_2 \wedge \Omega \\
d\,\text{Im}(\Upsilon) & = -2\tau_0\,\Omega^2 - J\tau_5 \wedge \text{Re}(\Upsilon) + \widehat{\tau}_2 \wedge \Omega
\end{align*}
together with certain algebraic conditions on $\tau_2, \widehat{\tau}_2, \tau_3$ that will be explained in $\S$\ref{ssect:prelim}. \\
\indent In order to study special Lagrangian $3$-folds in $M$, we will break the torsion forms into $\text{SO}(3)$-irreducible pieces with respect to a certain splitting of $TM$, where we regard $\text{SO}(3) \leq \text{SU}(3)$ as the stabilizer of a special Lagrangian $3$-plane.  Indeed, in $\S$2.3, we will decompose $\tau_0, \widehat{\tau}_0, \tau_2, \widehat{\tau}_2, \tau_3, \tau_4, \tau_5$ into $\text{SO}(3)$-irreducible components, writing
\begin{align*}
\tau_0 & = \tau_0 & \tau_2 & = (\tau_2)_1 + (\tau_2)_2 & \tau_3 & = (\tau_3)_0' + (\tau_3)_0'' + (\tau_3)_2' + (\tau_3)_2'' & \tau_4 & = (\tau_4)_{\mathsf{T}} + (\tau_4)_{\mathsf{N}} \\
\widehat{\tau}_0 & = \widehat{\tau}_0 & \widehat{\tau}_2 & = (\widehat{\tau}_2)_1 + (\widehat{\tau}_2)_2 & & & \tau_5 & = (\tau_5)_{\mathsf{T}} + (\tau_5)_{\mathsf{N}}
\end{align*}
We will refer to the individual pieces
$$\tau_0, \widehat{\tau}_0, \ \ \ (\tau_2)_1, (\tau_2)_2, (\widehat{\tau}_2)_1, (\widehat{\tau}_2)_2, \ \ \ (\tau_3)_0', (\tau_3)_0'', (\tau_3)_2', (\tau_3)_2'', \ \ \ (\tau_4)_{\mathsf{T}}, (\tau_4)_{\mathsf{N}}, \ \ \ (\tau_5)_{\mathsf{T}}, (\tau_5)_{\mathsf{N}}$$
as \textit{refined torsion forms} (with respect to a splitting of $TM$).  It turns out that the mean curvature of a special Lagrangian can be expressed purely in terms of the refined torsion forms: \\

\noindent \textbf{Theorem \ref{thm:sLagMC} (Mean Curvature):} Let $M^6$ be a $6$-manifold with an $\text{SU}(3)$-structure.  The mean curvature vector $H$ of a special Lagrangian $3$-fold of phase $\theta$ immersed in $M$ is given by
$$H = \textstyle -\frac{1}{\sqrt{2}}\cos(\theta)\,[ (\tau_2)_1 ]^\natural - \frac{1}{\sqrt{2}}\sin(\theta)\,[ (\widehat{\tau}_2)_1 ]^\natural  + \sin(\frac{\theta}{3})\,[(\tau_5)_{\mathsf{T}}]^\S - \cos(\frac{\theta}{3})\,[J(\tau_5)_{\mathsf{N}}]^\S$$
%%%%% PHASE EDIT MADE
where $\natural$, $\S$ are certain isometric isomorphisms defined in (\ref{eq:SU3-NatIsom}) and (\ref{eq:RotIsom}), respectively. \\
\indent In particular, an $\text{SU}(3)$-structure on $M$ has the property that every special Lagrangian $3$-fold (of every phase) in $M$ is minimal if and only if $\tau_2 = \widehat{\tau}_2 = \tau_5 = 0$. \\

\indent This formula can be regarded as a submanifold analogue of the curvature formulas derived by Bedulli and Vezzoni \cite{MR2287296} for $6$-manifolds with $\text{SU}(3)$-structures.  We will also derive an obstruction to the local existence of special Lagrangian $3$-folds In the language of refined torsion forms: \\

\noindent \textbf{Theorem \ref{thm:sLagObs} (Local Obstruction):} If a special Lagrangian $3$-fold $\Sigma$ of phase $\theta$ exists in $M$, then the following relation holds at points of $\Sigma$:
$$\widehat{\tau}_0 \sin(\theta) + \tau_0 \cos(\theta) = \textstyle \frac{\sqrt{3}}{6} \left(  \sin(\frac{\theta}{3})  [(\tau_3)_0'']^\ddag - \cos(\frac{\theta}{3}) [(\tau_3)_0']^\dagger \displaystyle \right)$$
where $\dagger, \ddag$ are certain isometric isomorphisms defined in (\ref{eq:SU3-DagIsom}) and (\ref{eq:SU3-DDagIsom}), respectively. \\
\indent In particular, if $\tau_3 = 0$, then the phase of every special Lagrangian $3$-fold in $M$ satisfies $\tan(\theta) = -\tau_0/\widehat{\tau}_0$. \\
%%%%% PHASE EDIT MADE

%%% Note that Theorem \ref{thm:sLagObs} generalizes the well-known fact that ........

\noindent \textbf{Corollary \ref{cor:sLagObs}:} Fix a point $x \in M$ and a phase $\theta \in [0, 2\pi)$. If every phase $\theta$ special Lagrangian $3$-plane in $T_xM$ is tangent to a phase $\theta$ special Lagrangian $3$-fold, then $\tau_3|_x = 0$ and $\widehat{\tau}_0|_x \sin(\theta) = -\tau_0|_x \cos(\theta)$.
%%%%% PHASE EDIT MADE

\subsection{Organization}

\indent \indent This work is organized in the following way.  In \S\ref{ssect:prelim}, we set our foundations, recalling the basic theory of $\text{SU}(3)$-structures on $6$-manifolds and special Lagrangian $3$-folds.  Since $\text{SO}(3) \leq \text{SU}(3)$ is the stabilizer of a special Lagrangian $3$-plane, the representation theory of $\text{SO}(3)$ plays a key role in the geometry of special Lagrangian $3$-folds.  Accordingly, we spend \S\ref{ssect:SO3reps} describing the relevant representations.  Most importantly, we will provide explicit descriptions of the $\text{SO}(3)$-irreducible components of $\Lambda^k(\mathbb{R}^6)$ and $\text{Sym}^2(\mathbb{R}^6)$. \\
\indent Then in \S\ref{ssect:sLagreftors}, we will use the machinery of \S\ref{ssect:SO3reps} to define the refined torsion forms and express them in terms of a local $\text{SO}(3)$-frame.  Finally, in \S\ref{ssect:sLagMC}, we prove Theorem \ref{thm:sLagObs}, Corollary \ref{cor:sLagObs}, and Theorem \ref{thm:sLagMC}.  After the proof of Theorem \ref{thm:sLagMC} concludes, we provide a table indicating how our results apply to a few classes of $\text{SU}(3)$-structures encountered in the literature.  \\

\noindent \textbf{Acknowledgements:} This work has benefited from conversations with Robert Bryant, Thomas Madsen, and Alberto Raffero.  The second author would also like to thank McKenzie Wang for his guidance and encouragement.  The first author thanks the Simons Collaboration on Special Holonomy in Geometry, Analysis and Physics for support during the period in which this article was written.

%In $\S$2.2, we will explain how to decompose the $\text{SU}(3)$-modules --- namely, $\Lambda^k(\mathbb{R}^6)$ and $\text{Sym}^2(\mathbb{R}^6$ --- that appear in the study of $\text{SU}(3)$-structures into $\text{SO}(3)$-irreducible submodules. \\
%\indent \indent In $\S$\ref{ssect:SO3reps}, we explain how to decompose the relevant $\text{SU}(3)$-modules (e.g., $\Lambda^k(\mathbb{R}^6)$ and $\text{Sym}^2(\mathbb{R}^6$)) that appear in the study of $\text{SU}(3)$-structures into $\text{SO}(3)$-irreducible pieces.  We will give explicit descriptions of these submodules, both for our own calculations and in the hopes that our setup will be useful to others. \\
% \indent Then, in $\S$2.3, we use the linear algebra of \S\ref{ssect:SO3reps} to define the relevant refined torsion forms, and express them in terms of a local $\text{SO}(3)$-frame.  Finally, in \S\ref{ssect:sLagMC}, we prove Theorem \ref{thm:sLagObs}, Corollary \ref{cor:sLagObs}, and Theorem \ref{thm:sLagMC}. \\

%%%%% Section 2.1 %%%%%

\section{Special Lagrangian $3$-Folds in $\text{SU}(3)$-Structures}

%\indent \indent Our goal in this section is to derive a formula (Theorem \ref{thm:sLagMC}) for the mean curvature of special Lagrangian $3$-folds of arbitrary phase in $6$-manifolds equipped with an $\text{SU}(3)$-structure.  In the process, we observe an obstruction (Theorem \ref{thm:sLagObs}) to their local existence. \\
%\indent These formulas and obstructions will be phrased in terms of \textit{refined torsion forms}, which we will define in \S\ref{sssect:sLagreftorsLocFrame}.  These refined forms are simply the $\text{SO}(3)$-irreducible pieces of the usual torsion forms $\tau_0, \widehat{\tau}_0, \ldots, \tau_5$ of a $\text{SU}(3)$-structure.  As such, we will use $\S$\ref{ssect:SO3reps} to describe the relevant $\text{SO}(3)$-representation theory needed to decompose $\tau_0, \widehat{\tau}_0, \ldots, \tau_5$.

\subsection{Preliminaries} \label{ssect:prelim}

\indent \indent In this section, we define both the ambient spaces (in $\S$2.1.2) and submanifolds (in $\S$2.1.3) of interest.  We also use this section to fix notation and clarify conventions.

\subsubsection{$\text{SU}(3)$-Structures on Vector Spaces}

\indent \indent Let $V = \mathbb{R}^6$ equipped with the standard inner product $\langle \cdot, \cdot \rangle$ and norm $\Vert \cdot \Vert$.  Let $\{e_1, \ldots, e_6\}$ denote the standard (orthonormal) basis of $V$, and let $\{e^1, \ldots, e^6\}$ denote the corresponding dual basis of $V^*$.  We will regard $V \simeq \mathbb{C}^3$ via the complex structure $J_0$ given by
\begin{align*}
J_0e_1 & = e_4 & J_0e_2 & = e_5 & J_0e_3 & = e_6.
\end{align*}
The standard symplectic form $\Omega_0 = \langle J_0\cdot, \cdot \rangle$ and complex volume form $\Upsilon_0$ are then given by
\begin{align*}
\Omega_0 & = e^{14} + e^{25} + e^{36} \\
\Upsilon_0 & = (e^1 + ie^4) \wedge (e^2 + ie^5) \wedge (e^3 + ie^6)
\end{align*}
where we use the shorthand $e^{ij} := e^i \wedge e^j$ and $e^{ijk} := e^i \wedge e^j \wedge e^k$. Note that $\Upsilon_0$ has real and imaginary parts
\begin{align*}
\text{Re}(\Upsilon_0) & = e^{123} - e^{156} + e^{246} - e^{345} \\
\text{Im}(\Upsilon_0) & = e^{126} - e^{135} + e^{234} - e^{456}
\end{align*}
Note also that
$$\text{vol}_0 := \textstyle \frac{1}{6}\,\Omega_0^3 = \frac{i}{8}\,\Upsilon_0 \wedge \overline{\Upsilon}_0 = e^{142536}$$
is a (real) volume form $V$. \\
\indent For calculations, it will be convenient to express $\Omega_0$ and $\Upsilon_0$ in the form
\begin{align}
\Omega_0 & = \textstyle \frac{1}{2}\Omega_{ij}\,e^{ij} & \text{Re}(\Upsilon_0) & = \textstyle \frac{1}{6}\epsilon_{ijk}\,e^{ijk} & \text{Im}(\Upsilon_0) & = \textstyle \frac{1}{6}\widehat{\epsilon}_{ijk}\,e^{ijk} \label{eq:OmegaUpsilon}
\end{align}
where the constants $\Omega_{ij}, \epsilon_{ijk}, \widehat{\epsilon}_{ijk} \in \{-1,0,1\}$ are defined by the above formulas.  For example, $\Omega_{14} = -\Omega_{41} = 1$ and $\epsilon_{123} = -\epsilon_{213} = 1$.  Identities involving the $\Omega$- and $\epsilon$-symbols are given in \cite{MR2287296}.

\begin{rmk}
	The data $\langle \cdot, \cdot \rangle, J_0, \Omega_0, \Upsilon_0$ are not independent of one another, and one can recover $\langle \cdot, \cdot \rangle$ and $J_0$ from the knowledge of $\Omega_0$ and $\Upsilon_0$.  Let us be more precise. \\
	\indent In general, suppose $(g,J,\Omega, \Upsilon)$ is a quadruple on $V$ consisting of a positive-definite inner product $g$, a complex structure $J$, a non-degenerate $2$-form $\Omega$ defined by $g = \Omega(\cdot, J \cdot)$, and a complex $(3,0)$-form $\Upsilon \in \Lambda^3(V^*; \mathbb{C})$ for which $\Upsilon \wedge \overline{\Upsilon} \neq 0$.  Then $\Upsilon$ is decomposable, satisfies $\Omega \wedge \Upsilon = 0$, the $6$-form $\frac{i}{8}\Upsilon \wedge \overline{\Upsilon}$ is a real volume form, and finally
	\begin{align} \label{eq:SU3-Compat1}
	\textstyle g(X,Y)\,\frac{i}{8}\,\Upsilon \wedge \overline{\Upsilon} & = \textstyle -\frac{1}{2}\,\iota_X(\Omega) \wedge \iota_Y(\text{Re}(\Upsilon)) \wedge \text{Re}(\Upsilon).
	\end{align}
	 Conversely, let $(\Omega, \Upsilon) \in \Lambda^2(V^*) \oplus \Lambda^3(V^*; \mathbb{C})$ be a pair consisting of a non-degenerate $2$-form $\Omega$ and a decomposable complex $3$-form $\Upsilon$ satisfying $\Upsilon \wedge \overline{\Upsilon} \neq 0$ and $\Omega \wedge \Upsilon = 0$. Then one can recover $(g,J)$ via
	\begin{subequations} \label{eq:SU3-Compat2}
		\begin{align} 
		\textstyle \iota_{JX}(\frac{i}{8}\,\Upsilon \wedge \overline{\Upsilon}) & = -\textstyle \frac{1}{2}\,\iota_X(\text{Re}(\Upsilon)) \wedge \text{Re}(\Upsilon)  \\
		g(X,Y) & = \Omega(X,JY).
		\end{align}
	\end{subequations}
	For a proof, see \cite{schulte2010half}. %$\Box$
\end{rmk}

% \indent [Definition of $\text{SU}(3)$ as corresponding stabilizer] \\

\indent We always equip $\Lambda^k(V^*)$ with the usual inner product, also denoted $\langle \cdot, \cdot \rangle$, given by declaring
\begin{equation} \label{eq:SU3-LambdaMetric}
\{e^I \colon I \text{ increasing multi-index}\}
\end{equation}
to be an orthonormal basis.  We let $\Vert \cdot \Vert$ denote the corresponding norm.  We will also need both the orthogonal and symplectic Hodge star operators.  These are the respective operators $\ast, \star \colon \Lambda^k(V^*) \to \Lambda^{6-k}(V^*)$ such that every $\alpha, \beta \in \Lambda^k(V^*)$ satisfy
\begin{align*}
\alpha \wedge \ast \beta & = \langle \alpha, \beta \rangle\,\text{vol}_0 & \alpha \wedge \star \beta & = \Omega_0(\alpha, \beta)\,\text{vol}_0
\end{align*}

\indent We view $V \simeq \mathbb{R}^6$ as the standard $\text{SU}(3)$-representation.  This $\text{SU}(3)$-representation is irreducible.  However, the induced $\text{SU}(3)$-representations on $\Lambda^k(V^*)$ for $2 \leq k \leq 5$ are not irreducible.  Indeed, $\Lambda^2(V^*)$ decomposes into irreducible $\text{SU}(3)$-modules as
$$\Lambda^2(V^*) = \mathbb{R}\,\Omega_0 \oplus \Lambda^2_6 \oplus \Lambda^2_8,$$
where
\begin{align*}
\Lambda^2_6 & = \{\iota_X(\text{Re}(\Upsilon_0)) \colon X \in V\} = \{\star(\alpha \wedge \text{Re}(\Upsilon_0)) \colon \alpha \in \Lambda^1\} \\
\Lambda^2_8 & = \{\beta \in \Lambda^2 \colon \beta \wedge \text{Re}(\Upsilon_0) = 0 \text{ and } \star\!\beta = -\beta \wedge \Omega_0\}
\end{align*}
Similarly, $\Lambda^3(V^*)$ decomposes into irreducible $\text{SU}(3)$-modules as
$$\Lambda^3(V^*) = \mathbb{R}\,\text{Re}(\Upsilon_0) \oplus \mathbb{R}\,\text{Im}(\Upsilon_0) \oplus \Lambda^3_6 \oplus \Lambda^3_{12}$$
where
\begin{align*}
\Lambda^3_6 & = \{\alpha \wedge \Omega_0 \colon \alpha \in \Lambda^1\} = \{\gamma \in \Lambda^3 \colon \star \!\gamma = \gamma\} \\
\Lambda^3_{12} & = \{\gamma \in \Lambda^3 \colon \gamma \wedge \Omega_0 = 0 \text{ and } \gamma \wedge \text{Re}(\Upsilon_0) = \gamma \wedge \text{Im}(\Upsilon_0) = 0\}.
\end{align*}
In each case, $\Lambda^k_\ell$ is an irreducible $\text{SU}(3)$-module of dimension $\ell$.  One can obtain similar decompositions of $\Lambda^4(V^*)$ and $\Lambda^5(V^*)$ by applying the orthogonal Hodge star operator. \\
\indent The $\text{SU}(3)$-module $\text{Sym}^2(V^*)$ is also reducible, splitting as
$$\text{Sym}^2(V^*) = \mathbb{R} \,\text{Id} \oplus \text{Sym}^2_+ \oplus \text{Sym}^2_-$$
where
\begin{align*}
\text{Sym}^2_+ & = \{h \in \text{Sym}^2(V^*) \colon J_0h = h,\, \text{tr}(h) = 0\} \\
\text{Sym}^2_- & = \{h \in \text{Sym}^2(V^*) \colon J_0h = -h\}.
\end{align*}
where here and in the rest of the paper we use the metric $g_0$ to identify $\text{Sym}^2(V^*)$ with the space of symmetric matrices. In \cite{MR2287296}, the authors note that the maps
\begin{align}
\rho \colon \text{Sym}^2_+ & \to \Lambda^2_8 & \chi \colon \text{Sym}^2_- & \to \Lambda^3_{12} \label{eq:rho-chi} \\
\rho(h_{ij}e^ie^j) & = h_{ik}\Omega_{kj}\,e^{ij} & \chi(h_{ij}e^i e^j) & = h_{i\ell} \epsilon_{\ell jk}\,e^{ijk} \notag
\end{align}
are $\text{SU}(3)$-module isomorphisms.  These isomorphisms will be crucial to our calculations in $\S$\ref{ssect:SO3reps}.

\subsubsection{$\text{SU}(3)$-Structures on $6$-Manifolds}

\begin{defn}
	Let $M$ be an oriented $6$-manifold.  An \textit{$\text{SU}(3)$-structure} on $M$ is a pair $(\Omega, \Upsilon)$ consisting of a non-degenerate $2$-form $\Omega \in \Omega^2(M)$ and a complex $3$-form $\Upsilon \in \Omega^3(M; \mathbb{C})$ such that at each $x \in M$, there exists a coframe $u \colon T_xM \to \mathbb{R}^6$ for which $\Omega|_x = u^*(\Omega_0)$ and $\Upsilon|_x = u^*(\Upsilon_0)$.
\end{defn}

Intuitively, an $\text{SU}(3)$-structure is a smooth identification of each tangent space $T_xM$ with $\mathbb{C}^3$ in such a way that $(\Omega|_x, \Upsilon|_x)$ is aligned with $(\Omega_0, \Upsilon_0)$.  We note that a $6$-manifold $M$ admits an $\text{SU}(3)$-structure if and only if it is orientable and spin \cite[Chapter 9]{LaMiSpin89}.

Every $\text{SU}(3)$-structure $(\Omega, \Upsilon)$ on $M$ induces a Riemannian metric $g$ and an almost-complex structure $J$ on $M$ via the formulas (\ref{eq:SU3-Compat1})-(\ref{eq:SU3-Compat2}), reflecting the inclusion $\text{SU}(3) \leq \text{SO}(6) \cap \text{GL}_3(\mathbb{C})$.  We emphasize that, in general, $J$ need not be integrable, and $\Omega$ need not be closed.  We also caution that the association $(\Omega, \Upsilon) \mapsto g$ is not injective. \\

The first-order local invariants of an $\text{SU}(3)$-structure are completely encoded in a certain $\text{SU}(3)$-equivariant function
$$T \colon F_{\text{SU}(3)} \to \Lambda^0 \oplus \Lambda^0 \oplus \Lambda^2_8 \oplus \Lambda^2_8 \oplus \Lambda^3_{12} \oplus \Lambda^1 \oplus \Lambda^1 \simeq \mathbb{R}^{42}$$
called the \textit{intrinsic torsion function}, defined on the total space of the $\text{SU}(3)$-frame bundle $F_{\text{SU}(3)} \to M$ over $M$.  We think of $T$ as describing the $1$-jet of the $\text{SU}(3)$-structure. \\
\indent The intrinsic torsion function is somewhat technical to define --- the interested reader can find more information in \cite{MR1062197} and \cite{MR1004008} --- but several equivalent reformulations are available.  Most conveniently for our purposes: the intrinsic torsion function of a $\text{SU}(3)$-structure is equivalent to the data of the $3$-form $d\Omega$ and the complex $4$-form $d\Upsilon$. \\
\indent In \cite{MR2287296}, the exterior derivatives of $\Omega$ and $\Upsilon$ are shown to take the form
\begin{align*}
d\Omega & = 3\tau_0\,\text{Re}(\Upsilon) + 3\widehat{\tau}_0\,\text{Im}(\Upsilon) + \tau_3 + \tau_4 \wedge \Omega \\
d\,\text{Re}(\Upsilon) & = 2\widehat{\tau}_0\,\Omega^2 + \tau_5 \wedge \text{Re}(\Upsilon) + \tau_2 \wedge \Omega \\
d\,\text{Im}(\Upsilon) & = -2\tau_0\,\Omega^2 - J\tau_5 \wedge \text{Re}(\Upsilon) + \widehat{\tau}_2 \wedge \Omega
\end{align*}
where
$$(\tau_0, \widehat{\tau}_0, \tau_2, \widehat{\tau}_2, \tau_3, \tau_4, \tau_5) \in \Gamma(\Lambda^0 \oplus \Lambda^0 \oplus \Lambda^2_8 \oplus \Lambda^2_8 \oplus \Lambda^3_{12} \oplus \Lambda^1 \oplus \Lambda^1)$$
and we are abbreviating $\Lambda^k_\ell := \Lambda^k_\ell(T^*M)$, etc.  We refer to $\tau_0, \widehat{\tau}_0, \tau_2, \widehat{\tau}_2, \tau_3, \tau_4, \tau_5$ as the \textit{torsion forms} of the $\text{SU}(3)$-structure.

Following standard conventions, we let $X_0^+, X_0^-, X_2^+, X_2^-, X_3, X_4, X_5$ denote the vector bundles $\Lambda^0, \Lambda^0, \Lambda^2_8, \Lambda^2_8, \Lambda^3_{12}, \Lambda^1, \Lambda^1$, respectively.  %Consider the set $\mathcal{S}$ consisting of the $2^7 = 128$ vector bundles
%$$\mathcal{S} = \left\{ 0, \,\bigoplus_{k=1}^\ell E_k \colon E_k \in \{X_0^+, X_0^-, X_2^+, X_2^-, X_3, X_4, X_5\}, \ell = 1, \ldots, 7 \right\}\!.$$
%\begin{defn}
%	Let $E \in \mathcal{S}$ be a vector bundle on the list above. We say that an $\text{SU}(3)$-structure belongs to the \textit{torsion class $E$} iff the torsion forms of the $\text{SU}(3)$-structure $(\tau_0, \widehat{\tau}_0, \tau_2, \widehat{\tau}_2, \tau_3, \tau_4, \tau_5) \in \Gamma(X_0^+ \oplus \cdots \oplus X_5)$ is valued in $E \subset X_0^+ \oplus \cdots \oplus X_5$.
%\end{defn}
In the sequel, we will say (for example) that an $\text{SU}(3)$-structure belongs to torsion class $X_0^+ \oplus X_0^- \oplus X_2^+ \oplus X_2^-$ if and only if $\tau_3 = \tau_4 = \tau_5 = 0$, etc.

\subsubsection{Special Lagrangian $3$-Folds}

\indent \indent Let $(M^6, \Omega, \Upsilon)$ be a $6$-manifold with an $\text{SU}(3)$-structure, and fix a tangent space $(T_xM, \Omega|_x, \Upsilon|_x) \simeq (V, \Omega_0, \Upsilon_0)$.  In their work on calibrations, Harvey and Lawson \cite{harvey1982calibrated} observed that the vector space $(V, \Omega_0, \Upsilon_0)$ possesses an $\mathbb{S}^1$-family of distinguished classes of $3$-dimensional subspaces --- the special Lagrangian $3$-planes of a given phase --- which we now describe.
% --- the special Lagrangian $3$-planes of phase $\theta$ (to be defined shortly) --- first studied by Harvey and Lawson [HL] in their work on calibrations. \\  %The vector space $(V, \Omega_0, \Upsilon_0)$ possesses a distinguished family of subspaces, namely the family $\text{Lag}_3 \subset \text{Gr}_3(\mathbb{R}^6)$ of Lagrangian $3$-planes.  Since $\text{U}(3)$ acts transitively on $\text{Lag}_3$ with stabilizer $\text{SO}(3)$ (see [HL]), we may identify
%$$\text{Lag}_3 \simeq \frac{\text{U}(3)}{\text{SO}(3)}.$$
% The normal subgroup $\text{SU}(3) \trianglelefteq \text{U}(3)$ acts on $\text{Lag}_3$ with cohomogeneity-one. \\

For $\theta \in [0,2\pi)$, consider the complex $3$-form $\Upsilon_\theta \in \Lambda^3(V^*; \mathbb{C})$ defined by
$$\Upsilon_\theta := e^{-i\theta}\Upsilon_0.$$
We refer to its real part
$$\text{Re}(\Upsilon_\theta) = \text{Re}(e^{-i\theta}\Upsilon_0) \in \Lambda^3(V^*)$$
as the \textit{phase $\theta$ special Lagrangian $3$-form}, following the sign convention of \cite{joyce2007riemannian} (rather than \cite{harvey1982calibrated}).  Note that $\text{Im}(\Upsilon_\theta) = \text{Re}(\Upsilon_{\theta + \frac{\pi}{2}})$, where $\theta + \frac{\pi}{2}$ is regarded mod $2\pi$.  The $3$-forms $\Upsilon_\theta$ enjoy the following remarkable property:

\begin{prop}[\cite{harvey1982calibrated}]\label{prop:sLagcomass}
	 For each $\theta \in [0,2\pi)$, the $3$-form $\mathrm{Re}(\Upsilon_\theta)$ has co-mass one, meaning that:
	$$\left|\mathrm{Re}(\Upsilon_\theta)(x,y,z)\right| \leq 1$$
	for every orthonormal set $\{x,y,z\}$ in $V \simeq \mathbb{R}^6$.
\end{prop}
In light of this proposition, it is natural to examine more closely those $3$-planes $E \in \text{Gr}_3(V)$ for which $\left|\text{Re}(\Upsilon_\theta)(E)\right| = 1$.
\begin{prop}[\cite{harvey1982calibrated}]\label{prop:sLagdefn}
	Let $E \in \text{Gr}_3(V)$ be a $3$-plane in $V$.  The following are equivalent: \\
	\indent (i) If $\{u,v,w\}$ is an orthonormal basis of $E$, then $\text{Re}(\Upsilon_\theta)(u,v,w) = \pm 1$. \\
	\indent (ii) $E$ is Lagrangian and $\left.\text{Im}(\Upsilon_\theta)\right|_E = 0$. \\
	\noindent If either of these conditions hold, we say that $E$ is \emph{special Lagrangian (SL) of phase $\theta$}.
\end{prop}

%%%%% PHASE EDIT MADE
Note that every Lagrangian $3$-plane is special Lagrangian for some phase $\theta$.  Note also that the $\mathbb{S}^1$-action on $V = \mathbb{R}^6 \simeq \mathbb{C}^3$ given by
$$e^{i\vartheta} \cdot (z_1, z_2, z_3) := (e^{i\vartheta}z_1, e^{i\vartheta}z_2, e^{i\vartheta}z_3),$$
induces a ``change-of-phase" $\mathbb{S}^1$-action on $\text{Lag}(V) = \{E \in \text{Gr}_3(V) \colon E \text{ Lagrangian}\}$.   Explicitly, letting $\{e_1, \ldots, e_6\}$ denote the standard $\mathbb{R}$-basis of $V$, and letting
\begin{subequations} \label{eq:vw-vectors} 
\begin{align}
v_1(\theta) &  = \textstyle \cos(\frac{\theta}{3}) e_1 + \sin(\frac{\theta}{3}) e_4 & w_1(\theta) & \textstyle = -\sin(\frac{\theta}{3}) e_1 + \cos(\frac{\theta}{3}) e_4 \\
v_2(\theta) & = \textstyle \cos(\frac{\theta}{3}) e_2 + \sin(\frac{\theta}{3}) e_5 & w_2(\theta) & \textstyle = -\sin(\frac{\theta}{3}) e_2 + \cos(\frac{\theta}{3}) e_5 \\
v_3(\theta) & = \textstyle \cos(\frac{\theta}{3}) e_3 + \sin(\frac{\theta}{3}) e_6 & w_3(\theta) & \textstyle = -\sin(\frac{\theta}{3}) e_3 + \cos(\frac{\theta}{3}) e_6 
\end{align}
\end{subequations}
the set $\{v_1(\theta), v_2(\theta), v_3(\theta)\}$ is an oriented basis for the SL $3$-plane $e^{i\theta/3} \cdot \text{span}(e_1, e_2, e_3)$ of phase $\theta$, and $\{w_1(\theta), w_2(\theta), w_3(\theta)\}$ is an oriented basis for the SL $3$-plane $e^{i\theta/3} \cdot \text{span}(e_4, e_5, e_6)$ of phase $\theta + \frac{3\pi}{2}$. \\

\indent Now, the $\text{SU}(3)$-action on $V$ induces an $\text{SU}(3)$-action on $\text{Gr}_3(V)$.  This action  on $\text{Gr}_3(V)$ is not transitive: for example, the subset of $\text{Gr}_3(V)$ consisting of special Lagrangian $3$-planes of a fixed phase $\theta$ is an $\text{SU}(3)$-orbit.  The corresponding stabilizer will play a crucial role in this section:

\begin{prop}[\cite{harvey1982calibrated}]
	Fix $\theta \in [0,2\pi)$. The Lie group $\textup{SU}(3)$ acts transitively on the subset of special Lagrangian $3$-planes of phase $\theta$:
	\begin{align*}
	\{E \in \textup{Gr}_3(V) \colon |\textup{Re}(\Upsilon_\theta)(E)| = 1\} \subset \textup{Gr}_3(V)
	\end{align*}
	The stabilizer of the $\textup{SU}(3)$-action is isomorphic to $\textup{SO}(3)$.
\end{prop}

We may finally define our primary objects of interest:

\begin{defn}
	Let $(M^6, \Omega, \Upsilon)$ be a $6$-manifold equipped with an $\text{SU}(3)$-structure $(\Omega, \Upsilon)$.  Identify each tangent space $(T_xM, \Omega|_x, \Upsilon|_x) \simeq (V, \Omega_0, \Upsilon_0)$.  Fix $\theta \in [0,2\pi)$. A \textit{special Lagrangian $3$-fold of phase $\theta$} in $M$ is a $3$-dimensional immersed submanifold $\Sigma \subset M$ for which each tangent space $T_x\Sigma \subset T_xM$ is a special Lagrangian $3$-plane of phase $\theta$.
\end{defn}
% \indent We emphasize that in this definition, each tangent space $T_x\Sigma$ is required to be special Lagrangian of the same fixed phase $\theta$. \\

Note that if $d(\text{Re}(\Upsilon)) = 0$, then $\text{Re}(\Upsilon)$ is a calibration whose calibrated $3$-planes are the special Lagrangian $3$-planes of phase $0$.  Thus, in this case, the phase 0 special Lagrangian $3$-folds are calibrated submanifolds, and hence are minimal submanifolds of $M$. Similarly, if $d(\text{Im}(\Upsilon)) = 0$, then $\text{Im}(\Upsilon)$ is a calibration, so the phase $\frac{\pi}{2}$ special Lagrangian $3$-folds are calibrated submanifolds of $M$.

%%%%% Section 2.2 %%%%%
\subsection{Some $\text{SO}(3)$-Representation Theory}\label{ssect:SO3reps}

\indent\indent Let the group $\text{SO}(3)$ act on $\R^3 = \text{span} \lbrace x, y, z \rbrace$ in the usual way. This action extends to an action of $\text{SO}(3)$ on the polynomial ring $\R [x,y,z]$. Let $\Vc_n \subset \R [x,y,z]$ be the $\text{SO}(3)$-submodule of homogeneous polynomials of degree $n$, and let $\Hc_n \subset \Vc_n$ denote the $\text{SO}(3)$-submodule of harmonic polynomials of degree $n$, i.e. the space of degree $n$ homogeneous polynomials $P$ satisfying $\Delta P = 0,$ an irreducible $\text{SO}(3)$-module of dimension $2n+1$. Every finite dimensional irreducible $\text{SO}(3)$-module is isomorphic to $\Hc_n$ for some $n$. \\
\indent The Clebsch-Gordan formula \cite[Chapter 5]{BrotD85} gives the decomposition of a tensor product of irreducible $\text{SO}(3)$-modules:
\begin{equation}\label{eq:ClebschGordon}
\Hc_a \otimes \Hc_b \cong \Hc_{a+b} \oplus \Hc_{a+b-1} \oplus \cdots \oplus \Hc_{\lvert a-b \rvert}.
\end{equation}
\noindent Only $\mathcal{H}_0$, $\mathcal{H}_1$, and $\mathcal{H}_2$ will play a role in this work.

\subsubsection{$\text{SO}(3)$ as a subgroup of $\text{SU}(3)$}

\indent \indent In our calculations, we shall need a concrete realization of $\text{SO}(3)$ as the stabilizer of a special Lagrangian plane. Let $\text{SO}(3)$ act on $V \cong \R^6$ via the identification $V \cong \Hc_1 \oplus \Hc_1,$ and let $e_1, \ldots, e_6$ be an orthonormal basis of $V$ such that:
\begin{itemize}
\item $\left\langle e_1, e_2, e_3 \right\rangle \cong \Hc_1$ and $\left\langle e_4, e_5, e_6 \right\rangle \cong \Hc_1$,
\item The map $e_i \mapsto e_{i+3}$ is $\text{SO}(3)$-equivariant.
\end{itemize}
Then the following forms are invariant under the $\text{SO}(3)$-action on $V$:
\begin{align*}
& e^{14} + e^{25} + e^{36}, \\
& (e^1 + ie^4) \wedge (e^2 + ie^5) \wedge (e^3 + ie^6).
\end{align*}
Thus, the action of $\text{SO}(3)$ on $V$ gives an embedding $\text{SO}(3) \subset \text{SU}(3)$. The 3-plane 
\begin{align*}
\left\langle v_1(\theta), v_2(\theta), v_3(\theta) \right\rangle = \textstyle\left\langle \cos (\frac{\theta}{3}) e_1 + \sin (\frac{\theta}{3}) e_4,\, \cos (\frac{\theta}{3}) e_2 + \sin (\frac{\theta}{3}) e_5 , \,\cos (\frac{\theta}{3}) e_3 + \sin (\frac{\theta}{3}) e_6 \right\rangle
\end{align*}
is special Lagrangian with phase $\theta$ and is preserved by the action of $\text{SO}(3)$.
%%%%% PHASE EDIT MADE

\subsubsection{Decomposition of the 1-forms on $V$}

\indent \indent Let $V \cong \Hc_1 \oplus \Hc_1$ be as in the previous section. The $\text{SO}(3)$-irreducible decomposition of the 1-forms on $V$ is given by
$$\Lambda^1(V^*) = \mathsf{T} \oplus \mathsf{N},$$
where
\begin{align*}
\mathsf{T} = \left\langle e^1, e^2, e^3 \right\rangle\!, \\
\mathsf{N} = \left\langle e^4, e^5, e^6 \right\rangle\!.
\end{align*}
As abstract $\text{SO}(3)$-modules, we have isomorphisms
\begin{align*}
\mathcal{H}_1 & \cong \mathsf{T} \cong \mathsf{N}  \cong \Lambda^2(\mathsf{T}) \cong \Lambda^2(\mathsf{N}) & \mathcal{H}_2 & \cong \text{Sym}^2_0(\mathsf{T}) \cong \text{Sym}^2_0(\mathsf{N}).
\end{align*}
\begin{defn}
	We let $\flat \colon V \to V^*$ via $X^\flat := \langle X, \cdot \rangle$ denote the usual (index-lowering) musical isomorphism, and let $\sharp \colon V^* \to V$ denote its inverse.  In the sequel, we let $\mathsf{T}^\sharp, \mathsf{N}^\sharp \subset V$ denote the images of $\mathsf{T}, \mathsf{N} \subset V^*$ under the $\sharp$ isomorphism.
	
	We also let $\S \colon \mathsf{T} \to e^{i\theta} \cdot \mathsf{N}^\sharp$ denote the map
	\begin{equation}
	\alpha \mapsto \alpha^\S = -\sin(\theta) \alpha^\sharp + \cos(\theta) J_0(\alpha^\sharp).
	\label{eq:RotIsom}
	\end{equation}
	Thus, for example, $(e^1)^\S = w_1(\theta)$, etc.
\end{defn} 
%%%%% PHASE EDIT MADE

\subsubsection{Decomposition of the Quadratic Forms on $V^*$}

\indent \indent We seek to decompose $\text{Sym}^2(V^*)$ into $\text{SO}(3)$-irreducible submodules.  One way to do this is to use $V^* = \mathsf{T} \oplus \mathsf{N}$ to split
\begin{align*}
\text{Sym}^2(V^*) & \cong (\mathbb{R}\text{Id}_{\mathsf{T}} \oplus \text{Sym}^2_0(\mathsf{T})) \oplus (\mathsf{T} \otimes \mathsf{N}) \oplus ( \mathbb{R}\text{Id}_{\mathsf{N}} \oplus \text{Sym}^2_0(\mathsf{N}) ) \\
& \cong (\mathbb{R}\text{Id}_{\mathsf{T}} \oplus \text{Sym}^2_0(\mathsf{T})) \oplus ( \mathcal{H}_0 \oplus \mathcal{H}_1 \oplus \mathcal{H}_2 ) \oplus ( \mathbb{R}\text{Id}_{\mathsf{N}} \oplus \text{Sym}^2_0(\mathsf{N}) )
\end{align*}
Alternatively, recall that $\text{Sym}^2(V^*)$ splits into $\text{SU}(3)$-irreducible submodules as
\begin{equation*}
\text{Sym}^2(V^*) \cong \mathbb{R}\,\text{Id} \oplus \text{Sym}^2_+ \oplus \text{Sym}^2_-
% \label{eq:Sym-Decomp}
\end{equation*}
Explicitly,
\begin{align*}
\text{Sym}^2_+ & = \left\{ \begin{bmatrix} h_2 & h_1 \\ -h_1 & h_2 \end{bmatrix} \colon h_1 \in \text{Skew}(\mathbb{R}^3), \ h_2 \in \text{Sym}^2_0(\mathbb{R}^3) \right\} \\
\text{Sym}^2_- & = \left\{ \begin{bmatrix} h' + c'\,\text{Id}_3 & h'' + c''\,\text{Id}_3 \\ h'' + c''\,\text{Id}_3 & -h' - c'\,\text{Id}_3 \end{bmatrix} \colon c', c'' \in \mathbb{R}, \ h', h'' \in \text{Sym}^2_0(\mathbb{R}^3) \right\}
\end{align*}
where $\text{Skew}(\mathbb{R}^3) \cong \Lambda^2(\mathbb{R}^3)$ denotes the vector space of skew-symmetric $3 \times 3$ matrices.  This description makes it plain that
\begin{equation}
\text{Sym}^2_+ = (\text{Sym}^2_+)_1 \oplus (\text{Sym}^2_+)_2 \label{eq:Sym2pldecomp}
\end{equation}
where we are defining
\begin{align*}
(\text{Sym}^2_+)_1 & := \left\{ \begin{bmatrix} 0 & h_1 \\ -h_1 & 0 \end{bmatrix} \colon h_1 \in \text{Skew}(\mathbb{R}^3) \right\} = \text{Sym}^2_+ \cap (\mathsf{T} \otimes \mathsf{N}) \cong \mathcal{H}_1 \\
(\text{Sym}^2_+)_2 &  := \left\{ \begin{bmatrix} h_2 & 0 \\ 0 & h_2 \end{bmatrix} \colon h_2 \in \text{Sym}^2_0(\mathbb{R}^3) \right\} = \text{Sym}^2_+ \cap (\text{Sym}^2_0(\mathsf{T}) \oplus \text{Sym}^2_0(\mathsf{T})) \cong \mathcal{H}_2
\end{align*}
Similarly, we see that 
$$\text{Sym}^2_- = (\text{Sym}^2_-)_0' \oplus (\text{Sym}^2_-)_2' \oplus (\text{Sym}^2_-)_0'' \oplus (\text{Sym}^2_-)_2''$$
where we are defining
\begin{align*}
(\text{Sym}^2_-)_0' & = \left\{ \begin{bmatrix} c'\,\text{Id}_3 & 0 \\ 0 & -c'\,\text{Id}_3 \end{bmatrix} \colon c' \in \mathbb{R} \right\} & (\text{Sym}^2_-)_0'' & = \left\{ \begin{bmatrix} 0 & c''\,\text{Id}_3 \\ c''\,\text{Id}_3 & 0 \end{bmatrix} \colon c'' \in \mathbb{R} \right\} \\
(\text{Sym}^2_-)_2' & = \left\{ \begin{bmatrix} h' & 0 \\ 0 & -h' \end{bmatrix} \colon h' \in \text{Sym}^2_0(\mathbb{R}^3) \right\} & (\text{Sym}^2_-)_2'' & = \left\{ \begin{bmatrix} 0 & h'' \\ h''  & 0 \end{bmatrix} \colon h', h'' \in \text{Sym}^2_0(\mathbb{R}^3) \right\}
\end{align*}

\subsubsection{Decomposition of the $2$-forms on $V$}

\indent \indent We now seek to decompose $\Lambda^2(V^*)$ into $\text{SO}(3)$-irreducible submodules.  As noted above, $\Lambda^2(V^*)$ splits into $\text{SU}(3)$-irreducible submodules as
\begin{align}
\Lambda^2(V^*) & \cong \mathbb{R}\Omega_0 \oplus \Lambda^2_6 \oplus \Lambda^2_8 \label{eq:SU3-2forms-1}
\end{align}
On the other hand, using $V^* = \mathsf{T} \oplus \mathsf{N}$, we may also decompose $\Lambda^2(V^*)$ as
\begin{align}
\Lambda^2(V^*) & \cong \Lambda^2(\mathsf{T}) \oplus (\mathsf{T} \otimes \mathsf{N}) \oplus \Lambda^2(\mathsf{N}). \label{eq:SU3-2forms-2}
\end{align}
We will refine both decompositions (\ref{eq:SU3-2forms-1}) and (\ref{eq:SU3-2forms-2}) into $\text{SO}(3)$-submodules.

To begin, note first that as $\text{SO}(3)$-modules, we have that $\mathbb{R}\Omega_0 \cong \mathcal{H}_0$ and $\Lambda^2(\mathsf{T}) \cong \mathcal{H}_1$ and $\Lambda^2(\mathsf{N}) \cong \mathcal{H}_1$ are irreducible.  Thus, it remains only to decompose $\Lambda^2_6$, $\Lambda^2_8$, and $\mathsf{T} \otimes \mathsf{N}$.
\begin{defn}
	Recall the isomorphism $\rho \colon \text{Sym}^2_+ \to \Lambda^2_8$ defined in (\ref{eq:rho-chi}).  We define
	\begin{align*}
	(\Lambda^2_6)_{\mathsf{T}} & = \{\iota_X(\text{Re}(\Upsilon_0)) \colon X \in \mathsf{T}^\sharp\} & (\Lambda^2_8)_1 & = \rho((\text{Sym}^2_+)_1) \\
	(\Lambda^2_6)_{\mathsf{N}} & = \{\iota_X(\text{Re}(\Upsilon_0)) \colon X \in \mathsf{N}^\sharp\} & (\Lambda^2_8)_2 & = \rho((\text{Sym}^2_+)_2)
	\end{align*}
	and
	$$(\mathsf{T} \otimes \mathsf{N})_1 = \{\alpha_1 \wedge \alpha_2 + J_0\alpha_2 \wedge J_0\alpha_1 \colon \alpha_1 \in \mathsf{T}, \, \alpha_2 \in \mathsf{N}\}.$$
\end{defn}
\begin{lem}\label{lem:sLagLam2}
	There exist decompositions
	\begin{align}
	\Lambda^2_6 & = (\Lambda^2_6)_{\mathsf{T}} \oplus (\Lambda^2_6)_{\mathsf{N}} \label{eq:Lam26decomp} \\
	\Lambda^2_8 & = (\Lambda^2_8)_1 \oplus (\Lambda^2_8)_2 \label{eq:Lam28decomp} \\
	\mathsf{T} \otimes \mathsf{N} & = \mathbb{R}\Omega_0 \oplus (\mathsf{T} \otimes \mathsf{N})_1 \oplus (\Lambda^2_8)_2 \label{eq:TNdecomp}
	\end{align}
	and these consist of $\text{SO}(3)$-irreducible submodules. \\
	\indent Thus, the decomposition
	\begin{align*}
	\Lambda^2(V^*) & = \mathbb{R}\Omega_0 \oplus \left[ (\Lambda^2_6)_{\mathsf{T}} \oplus (\Lambda^2_6)_{\mathsf{N}} \right] \oplus \left[ (\Lambda^2_8)_1 \oplus (\Lambda^2_8)_2 \right]
	\end{align*}
	is $\text{SO}(3)$-irreducible and refines (\ref{eq:SU3-2forms-1}), while
	\begin{align*}
	\Lambda^2(V^*) & = \Lambda^2(\mathsf{T}) \oplus \left[ \mathbb{R}\Omega_0 \oplus (\mathsf{T} \otimes \mathsf{N})_1 \oplus (\Lambda^2_8)_2 \right] \oplus \Lambda^2(\mathsf{N})
	\end{align*}
	is $\text{SO}(3)$-irreducible and refines (\ref{eq:SU3-2forms-2}).
\end{lem}
\begin{proof}
	The decomposition (\ref{eq:Lam26decomp}) follows from the isomorphism $V \to \Lambda^2_6, \: X \mapsto \iota_X\! \left( \text{Re}\! \left( \Upsilon_0 \right) \right)$ and the irreducible decomposition $V \cong \mathsf{T} \oplus \mathsf{N}$. \\
	\indent Decomposition (\ref{eq:Lam28decomp}) follows from applying the isomorphism $\rho \colon \text{Sym}^2_+ \to \Lambda^2_8$ to the irreducible decomposition (\ref{eq:Sym2pldecomp}) of $\text{Sym}^2_+$. \\
	\indent For decomposition (\ref{eq:TNdecomp}), note that as an $\text{SO}(3)$-module
	\begin{align}
	\mathsf{T} \otimes \mathsf{N} \cong \Hc_1 \otimes \Hc_1 \cong \Hc_0 \oplus \Hc_1 \oplus \Hc_2. \label{eq:TNSO3decomp}
	\end{align}
	The only trivial $\text{SO}(3)$-module contained in $\Lambda^2 \left( V^* \right)$ is $\R \Omega_0,$ so this must correspond to the trivial component of (\ref{eq:TNSO3decomp}). Similarly, the only $\text{SO}(3)$-module isomorphic to $\Hc_2$ contained in $\Lambda^2 \left( V^* \right)$ is $(\Lambda^2_8)_2,$ so this must correspond to the $\Hc_2$ component of (\ref{eq:TNSO3decomp}). The inclusion $\left( \mathsf{T} \otimes \mathsf{N} \right)_1 \subset \mathsf{T} \otimes \mathsf{N}$ is clear by construction, and since $\left( \mathsf{T} \otimes \mathsf{N} \right)_1 \cong \Hc_1$ we have demonstrated that decomposition (\ref{eq:TNdecomp}) holds.
\end{proof}

\begin{defn}
	Recall the isomorphism $\rho \colon \text{Sym}^2_+ \to \Lambda^2_8$ defined in (\ref{eq:rho-chi}) and the set $\{w_1(\theta), w_2(\theta), w_3(\theta)\}$ defined in (\ref{eq:vw-vectors}).  Consider the isomorphisms of $\text{SO}(3)$-modules given by
	\begin{align*}
	e^{i\theta} \cdot \mathsf{N}^\sharp & \to \text{Skew}(\mathbb{R}^3)  \to (\Lambda^2_8)_1 \\
	a_p w_p(\theta) & \mapsto h = \begin{bmatrix} 0 & a_3 & -a_2 \\ -a_3 & 0 & a_1 \\ a_2 & -a_1 & 0 \end{bmatrix} \mapsto \frac{1}{\sqrt{2}}\, \rho\! \left( \begin{bmatrix} 0 & h \\ h & 0 \end{bmatrix} \right)
	\end{align*}
	We will let
	\begin{equation}
	\natural \colon (\Lambda^2_8)_1 \to e^{i\theta} \cdot \mathsf{N}^\sharp
	\label{eq:SU3-NatIsom}
	\end{equation}
	denote the inverse of this isomorphism.  This map is, in fact, an isometry with respect to the given inner products on $e^{i\theta}\cdot \mathsf{N}^\sharp$ and $(\Lambda^2_8)_1$, due to the factor of $\frac{1}{\sqrt{2}}$.
\end{defn}

\subsubsection{Decomposition of the $3$-forms on $V$}

\indent \indent We now seek to decompose $\Lambda^3(V^*)$ into $\text{SO}(3)$-irreducible submodules.  As noted above, $\Lambda^3(V^*)$ splits into $\text{SU}(3)$-irreducible submodules as
\begin{align}
\Lambda^3(V^*) & \cong \mathbb{R}\,\text{Re}(\Upsilon_0) \oplus \mathbb{R}\,\text{Im}(\Upsilon_0) \oplus \Lambda^3_6 \oplus \Lambda^3_{12} \label{eq:SU3-3forms-1}
\end{align}
On the other hand, using $V^* = \mathsf{T} \oplus \mathsf{N}$, we may also decompose $\Lambda^2(V^*)$ as
\begin{align}
\Lambda^3(V^*) & \cong \Lambda^3(\mathsf{T}) \oplus (\Lambda^2(\mathsf{T}) \otimes \mathsf{N}) \oplus (\mathsf{T} \otimes \Lambda^2(\mathsf{N})) \oplus \Lambda^3(\mathsf{N}). \label{eq:SU3-3forms-2}
\end{align}
We will refine (\ref{eq:SU3-3forms-1}) into $\text{SO}(3)$-submodules.  To begin, note first that $\mathbb{R}\text{Re}(\Upsilon_0) \cong \mathbb{R}\text{Im}(\Upsilon_0) \cong \mathcal{H}_0$ are irreducible as $\text{SO}(3)$-modules, while $\Lambda^3_6$ and $\Lambda^3_{12}$ are not.

\begin{defn}
	Recall the isomorphism $\chi \colon \text{Sym}^2_- \to \Lambda^3_{12}$ of (\ref{eq:rho-chi}). We define
	\begin{align*}
	(\Lambda^3_6)_{\mathsf{T}} & = \{\alpha \wedge \Omega_0 \colon \alpha \in \mathsf{T}\} & (\Lambda^3_{12})_{0}' & = \chi( (\text{Sym}^2_-)_0' ) & (\Lambda^3_{12})_{0}'' & = \chi( (\text{Sym}^2_-)_0'' ) \\
	(\Lambda^3_6)_{\mathsf{N}} & = \{\alpha \wedge \Omega_0 \colon \alpha \in \mathsf{N}\} & (\Lambda^3_{12})_{2}' & = \chi( (\text{Sym}^2_-)_2' ) & (\Lambda^3_{12})_{2}'' & = \chi( (\text{Sym}^2_-)_2'' )
	\end{align*}
\end{defn}

\begin{lem}\label{lem:sLagLam3}
	The decompositions
	\begin{align*}
	\Lambda^3_6 & = (\Lambda^3_6)_{\mathsf{T}} \oplus (\Lambda^3_6)_{\mathsf{N}} \\
	\Lambda^3_{12} & = (\Lambda^3_{12})_{0}' \oplus (\Lambda^3_{12})_{2}' \oplus (\Lambda^3_{12})_{0}'' \oplus (\Lambda^3_{12})_{2}''
	\end{align*}
	consist of $\text{SO}(3)$-irreducible submodules.
\end{lem}

\begin{defn}
	We define maps $\dagger \colon (\Lambda^3_{12})_0' \to \mathbb{R}$ and $\
	\ddag \colon (\Lambda^3_{12})_0'' \to \mathbb{R}$ to be the unique vector space isomorphisms for which
	\begin{align}
	(-\text{Re}(\Upsilon_0) + 4e^{123})^\dagger & = 2\sqrt{3} \label{eq:SU3-DagIsom} \\
	(-\text{Im}(\Upsilon_0) - 4e^{456})^\ddag & = 2\sqrt{3} \label{eq:SU3-DDagIsom}
	\end{align}
	These maps are isometries (due to the choice of $2\sqrt{3}$) with respect to our inner product (\ref{eq:SU3-LambdaMetric}).
\end{defn}
\begin{rmk}
	To refine (\ref{eq:SU3-3forms-2}) into $\text{SO}(3)$-irreducible submodules, one simply has to decompose $\Lambda^2(\mathsf{T}) \otimes \mathsf{N}$ and $\Lambda^2(\mathsf{N}) \otimes \mathsf{T}$ into irreducibles.  This can be done by, say, tracing through the isomorphisms
	$$\Lambda^2(\mathsf{T}) \otimes \mathsf{N} \cong \mathsf{T} \otimes \mathsf{N} \cong \mathsf{T} \otimes \mathsf{T} \cong \mathbb{R} \oplus \text{Sym}^2_0(\mathsf{T}) \oplus \Lambda^2(\mathsf{T})$$
	and similarly for $\Lambda^2(\mathsf{N}) \otimes \mathsf{T}$.  Since we will not need such a refinement for this work, we leave the details to the interested reader.
\end{rmk}

%%%%% Section 2.3 %%%%%

\subsection{The Refined Torsion Forms} \label{ssect:sLagreftors}

\indent \indent Let $(M^6, \Omega, \Upsilon)$ be a $6$-manifold equipped with an $\text{SU}(3)$-structure $(\Omega, \Upsilon)$.  Fix a point $x \in M$, choose an arbitrary phase 0 special Lagrangian $3$-plane $\mathsf{T}^\sharp \subset T_xM$, and let $\mathsf{N}^\sharp \subset T_xM$ denote its orthogonal $3$-plane.  Our purpose in this section is to understand how the torsion of the $\text{SU}(3)$-structure decomposes with respect to the splitting
$$T_xM = \mathsf{T}^\sharp \oplus \mathsf{N}^\sharp.$$
\indent In \S\ref{sssect:sLagreftorsLocFrame}, we use Lemmas \ref{lem:sLagLam2} and \ref{lem:sLagLam3} to break the torsion forms $\tau_0, \widehat{\tau}_0, \ldots, \tau_5$ into $\text{SO}(3)$-irreducible pieces called \textit{refined torsion forms}.  Separately, in \S\ref{sssect:sLagTorsFuns}, we set up the $\text{SU}(3)$-coframe bundle $\pi \colon F_{\text{SU}(3)} \to M$ following \cite{MR2287296}, repackaging the original $\text{SU}(3)$ torsion forms $\tau_0, \widehat{\tau}_0, \ldots, \tau_5$ as a pair of functions
\begin{align*}
T = (T_{ij}) \colon F_{\text{SU}(3)} & \to \text{Mat}_{6 \times 6}(\mathbb{R}) \simeq \mathbb{R}^{36} \\
U = (U_i) \colon F_{\text{SU}(3)} & \to \mathbb{R}^6
\end{align*}
Finally, in \S\ref{sssect:sLagTorsDecomp}, we express the functions $T_{ij}$ and $U_i$ in terms of the (pullbacks of the) refined torsion forms.

\subsubsection{The Refined Torsion Forms in a Local $\text{SO}(3)$-Frame}\label{sssect:sLagreftorsLocFrame}

\indent \indent Fix $x \in M$ and split $T_x^*M = \mathsf{T} \oplus \mathsf{N}$ as above.  All of our calculations in this subsection will be done pointwise, and we will frequently suppress reference to $x \in M$.  By Lemmas 3.5 and 3.6, the torsion forms decompose into $\text{SO}(3)$-irreducible pieces as follows:
\begin{subequations} \label{eq:SU3-TorsionSplit-1}
\begin{align}
\tau_0 & = \tau_0 & \tau_2 & = (\tau_2)_1 + (\tau_2)_2 & \tau_3 & = (\tau_3)_0' + (\tau_3)_0'' + (\tau_3)_2' + (\tau_3)_2'' & \tau_4 & = (\tau_4)_{\mathsf{T}} + (\tau_4)_{\mathsf{N}} \\
\widehat{\tau}_0 & = \widehat{\tau}_0 & \widehat{\tau}_2 & = (\widehat{\tau}_2)_1 + (\widehat{\tau}_2)_2 & & & \tau_5 & = (\tau_5)_{\mathsf{T}} + (\tau_5)_{\mathsf{N}}
\end{align}
\end{subequations}
where here
\begin{subequations} 
\begin{align*}
(\tau_2)_1, (\widehat{\tau}_2)_1 & \in (\Lambda^2_8)_1 & (\tau_3)_0' & \in (\Lambda^3_{12})_0' & (\tau_3)_2' & \in (\Lambda^3_{12})_2' & (\tau_4)_{\mathsf{T}}, (\tau_5)_{\mathsf{T}} & \in \mathsf{T}  \\
(\tau_2)_2, (\widehat{\tau}_2)_2 & \in (\Lambda^2_8)_2 & (\tau_3)_0'' & \in (\Lambda^3_{12})_0'' & (\tau_3)_2'' & \in (\Lambda^3_{12})_2'' & (\tau_4)_{\mathsf{N}}, (\tau_5)_{\mathsf{N}} & \in \mathsf{N}
\end{align*}
\end{subequations}
We refer to $\tau_0, \widehat{\tau}_0, (\tau_2)_1, \ldots, (\tau_5)_{\mathsf{N}}$ as the \textit{refined torsion forms} of the $\text{SU}(3)$-structure at $x$ \textit{relative to the splitting $T_x^*M = \mathsf{T} \oplus \mathsf{N}$}. \\

%\begin{align*}
%(\tau_2)_1, (\widehat{\tau}_2)_1 & \in (\Lambda^2_8)_1 & (\tau_3)_0' & \in (\Lambda^3_{12})_0' & (\tau_4)_{\mathsf{T}}, (\tau_5)_{\mathsf{T}} & \in \mathsf{T}  \\
%(\tau_2)_2, (\widehat{\tau}_2)_2 & \in (\Lambda^2_8)_2 & (\tau_3)_0'' & \in (\Lambda^3_{12})_0'' & (\tau_4)_{\mathsf{N}}, (\tau_5)_{\mathsf{N}} & \in \mathsf{N} \\
% & & (\tau_3)_2' & \in (\Lambda^3_{12})_2' & \\
% & & (\tau_3)_2'' & \in (\Lambda^3_{12})_2'' &
%\end{align*}

\indent We seek to express the refined torsion in terms of a local $\text{SO}(3)$-frame. To that end, let $\{e_1, \ldots, e_6\}$ be an orthonormal basis for $T_xM$ for which $\mathsf{T}^\sharp = \text{span}(e_1, e_2, e_3)$ and $\mathsf{N}^\sharp = \text{span}(e_4, e_5, e_6)$.  Let $\{e^1, \ldots, e^6\}$ denote the dual basis for $T_x^*M$. \\

\noindent \textbf{Index Ranges:} We will employ the following index ranges: $1 \leq p,q \leq 3$ and $4 \leq \alpha, \beta \leq 6$ and $1 \leq i,j,k,\ell,m \leq 6$ and $1 \leq \delta \leq 5$.

\begin{defn}
	Define the $2$-forms
	\begin{align*}
	\Gamma_1 & = -e^{23} - e^{56} & \Upsilon_1 & = e^{26} + e^{35} & \Upsilon_4 & = e^{14} - e^{25} \\
	\Gamma_2 & = -e^{31} - e^{64} & \Upsilon_2 & = e^{16} + e^{34} & \Upsilon_5 & = e^{25} - e^{36} \\
	\Gamma_3 & = -e^{12} - e^{45} & \Upsilon_3 & = e^{15} + e^{24}
	\end{align*}
	These $2$-forms were obtained by applying $\rho \colon (\text{Sym}^2_+)_1 \oplus (\text{Sym}^2_+)_2 \to (\Lambda^2_8)_1 \oplus (\Lambda^2_8)_2$ of (\ref{eq:rho-chi}) to a suitable basis of $\text{Sym}^2_+$.
\end{defn}

\begin{lem}\label{lem:sLagLam2Basis}
	We have that:
	\begin{enumerate}[label=(\alph*),nosep]
		\item $\{\Gamma_1, \Gamma_2, \Gamma_3\}$ is a basis of $(\Lambda^2_8)_1.$
		\item $\{\Upsilon_1, \Upsilon_2, \Upsilon_3, \Upsilon_4, \Upsilon_5\}$ is a basis of $(\Lambda^2_8)_2.$
	\end{enumerate}
\end{lem}
\begin{defn}
	Define the $3$-forms
	\begin{align*}
	\Theta_1 & = -e^{245} - e^{364} & \Theta_4 & = -e^{156} + e^{264} & \Delta_1 & = e^{125} + e^{316} & \Delta_4 & = -e^{315} + e^{234} \\
	\Theta_2 & = -e^{145} - e^{356} & \Theta_5 & = e^{264} + e^{345} & \Delta_2 & = e^{124} + e^{236}  & \Delta_5 & = -e^{126} + e^{315} \\
	\Theta_3 & = -e^{164} - e^{256} & & & \Delta_3 & = e^{314} + e^{235}
	\end{align*}
	and
	\begin{align*}
	\Theta_0 & = -\text{Re}(\Upsilon_0) + 4e^{123} & \Delta_0 & = - \text{Im}(\Upsilon_0) - 4e^{456}
	\end{align*}
	These $3$-forms were obtained by applying the isomorphism
	$$\chi \colon (\text{Sym}^2_-)_0' \oplus (\text{Sym}^2_-)_0'' \oplus (\text{Sym}^2_-)_2' \oplus (\text{Sym}^2_-)_2'' \to (\Lambda^3_{12})_0' \oplus (\Lambda^3_{12})_0'' \oplus (\Lambda^3_{12})_2' \oplus (\Lambda^3_{12})_2''$$
	of (\ref{eq:rho-chi}) to a suitable basis of $\text{Sym}^2_-$.
\end{defn}

\begin{lem}\label{lem:sLagLam3Basis}
	 We have that:
	 \begin{enumerate}[label=(\alph*),nosep]
	 	\item $\{\Theta_0\}$ is a basis of $(\Lambda^3_{12})_0'$
	 	\item $\{\Delta_0\}$ is a basis of $(\Lambda^3_{12})_0''$
	 	\item $\{\Theta_1, \ldots, \Theta_5\}$ is a basis of $(\Lambda^3_{12})_2'$
	 	\item $\{\Delta_1, \ldots, \Delta_5\}$ is a basis of $(\Lambda^3_{12})_2''$
	 \end{enumerate}
\end{lem}

We now express $(\tau_2)_{\mathsf{T}}, (\widehat{\tau}_2)_{\mathsf{T}}, \ldots, (\tau_5)_{\mathsf{N}}$ in terms of the above bases.  That is, we define functions $A_p, B_\delta, C_p, D_\delta$ and $E_\delta, E_0, F_\delta, F_0$ and $G_p, J_p, M_p, N_p$ by:
\begin{subequations}  \label{eq:SU3-TorsionSplit-2}
\begin{align}
(\tau_2)_1 & = 4A_p\,\Gamma_p & (\tau_3)_0'  & =  4E_0\,\Theta_0  & (\tau_4)_{\mathsf{T}} & = 12G_p\,\omega^p \\
(\tau_2)_2 & = 4B_\delta\,\Upsilon_\delta & (\tau_3)_2'  & = 4E_\delta\,\Theta_\delta &  (\tau_4)_{\mathsf{N}} & = 12J_p\,\omega^{p+3} \\
(\widehat{\tau}_2)_1 & = 4C_p\,\Gamma_p & (\tau_3)_0'' & = 4F_0\,\Delta_0 & (\tau_5)_{\mathsf{T}} & = 3M_p\,\omega^p \\
(\widehat{\tau}_2)_2 & = 4D_\delta\,\Upsilon_\delta & (\tau_3)_2'' & =  4F_\delta\,\Delta_\delta & (\tau_5)_{\mathsf{N}} & = 3N_p\,\omega^{p+3}
\end{align}
\end{subequations}
The various factors of $3$, $4$, and $12$ are included simply for the sake of clearing future denominators.

\indent Note that the bases of Lemmas \ref{lem:sLagLam2Basis} and \ref{lem:sLagLam3Basis} are orthogonal but \textit{not} orthonormal with respect to the inner product (\ref{eq:SU3-LambdaMetric}) on $\Lambda^k(V^*)$.  Indeed, we have:
\begin{align*}
\Vert \Gamma_p \Vert & = \sqrt{2} & \Vert \Theta_\delta \Vert & = \sqrt{2} & \Vert \Theta_0 \Vert & = 2\sqrt{3} \\
\Vert \Upsilon_\delta \Vert & = \sqrt{2} & \Vert \Delta_\delta \Vert & = \sqrt{2} & \Vert \Delta_0 \Vert & = 2\sqrt{3}
\end{align*}
Thus, in terms of the isometric isomorphisms (\ref{eq:RotIsom}), (\ref{eq:SU3-NatIsom}), (\ref{eq:SU3-DagIsom}), (\ref{eq:SU3-DDagIsom}) of $\S$\ref{ssect:SO3reps}, we have:
\begin{subequations} \label{eq:SU3-Isom}
\begin{align}
[(\tau_3)_0']^\dagger & = 8\sqrt{3}\,E_0  & [(\tau_2)_1]^\natural  & = 4\sqrt{2}\,A_p w_p(\theta) & [(\tau_5)_{\mathsf{T}}]^\S & = 3M_p w_p(\theta) \\
[(\tau_3)_0'']^\ddag & = 8\sqrt{3}\,F_0  & [(\widehat{\tau}_2)_1]^\natural & = 4\sqrt{2}\,C_p w_p(\theta)  & [J(\tau_5)_{\mathsf{N}}]^\S & = 3N_p w_p(\theta)
\end{align}
\end{subequations}
We will need these for our calculations in \S\ref{ssect:sLagMC}.

\subsubsection{The Torsion Functions $T_{ij}$ and $U_i$}\label{sssect:sLagTorsFuns}

\indent \indent Let $(M^6, \Omega, \Upsilon)$ be a $6$-manifold with an $\text{SU}(3)$-structure $(\Omega, \Upsilon)$, and let $g$ denote the underlying Riemannian metric.  Let $F_{\text{SO}(6)} \to M$ denote the oriented orthonormal coframe bundle of $g$, and let $\omega = (\omega^1, \ldots, \omega^6) \in \Omega^1(F_{\text{SO}(6)}; \mathbb{R}^6)$ denote the tautological $1$-form.  By the Fundamental Lemma of Riemannian Geometry, there exists a unique $1$-form $\psi \in \Omega^1(F_{\text{SO}(6)}; \mathfrak{so}(6))$, the Levi-Civita connection form of $g$, satisfying the \textit{first structure equation}
\begin{equation*}
d\omega = -\psi \wedge \omega.
\end{equation*}
\indent Let $\pi \colon F_{\text{SU}(3)} \to M$ denote the $\text{SU}(3)$-coframe bundle of $M$.  Restricted to $F_{\text{SU}(3)} \subset F_{\text{SO}(6)}$, the Levi-Civita $1$-form $\psi$ is no longer a connection $1$-form in general.  Indeed, according to the splitting $\mathfrak{so}(6) = \mathfrak{su}(3) \oplus \mathbb{R}^6 \oplus \mathbb{R}$, we have the decomposition
$$\psi = \gamma + \lambda + \mu,$$
where $\gamma = (\gamma_{ij}) \in \Omega^1(F_{\text{SU}(3)}; \mathfrak{su}(3))$ is a connection $1$-form (the so-called \textit{natural connection} of the $\text{SU}(3)$-structure) and $\lambda \in \Omega^1(F_{\text{SU}(3)}; \mathbb{R}^6)$ and $\mu \in \Omega^1(F_{\text{SU}(3)}; \mathbb{R})$ are $\pi$-semibasic $1$-forms.  Here, we are viewing
\begin{align*}
\mathbb{R}^6 & \simeq \{ (\epsilon_{ijk}v_k) \in \mathfrak{so}(6) \colon (v_i) \in \mathbb{R}^6  \} \\
\mathbb{R} & \simeq \{ (a\Omega_{ij}) \in \mathfrak{so}(6) \colon a \in \mathbb{R} \}
\end{align*}
so that $\lambda$ and $\mu$ take the form
\begin{align*}
\lambda & = \left[ \begin{array}{c c c | c c c}
0 & \lambda_3 & -\lambda_2 & 0 & -\lambda_6 & \lambda_5 \\
-\lambda_3 & 0 & \lambda_1 & \lambda_6 & 0 & -\lambda_4 \\
\lambda_2 & -\lambda_1 & 0 & -\lambda_5 & \lambda_4 & 0 \\ \hline
0 & -\lambda_6 & \lambda_5 & 0 & -\lambda_3 & \lambda_2 \\
\lambda_6 & 0 & -\lambda_4 & \lambda_3 & 0 & -\lambda_1 \\
-\lambda_5 & \lambda_4 & 0 & -\lambda_2 & \lambda_1 & 0
\end{array}\right] & \mu & = \left[ \begin{array}{c c c | c c c}
0 & 0 & 0 & \mu & 0 & 0  \\
0 & 0 & 0 & 0 & \mu & 0  \\
0 & 0 & 0 & 0 & 0 & \mu  \\ \hline
-\mu & 0 & 0 & 0 & 0 & 0  \\
0 & -\mu & 0 & 0 & 0 & 0  \\
0 & 0 & -\mu & 0 & 0 & 0 
\end{array}\right]\!.
\end{align*}
Since $\lambda$ and $\mu$ are $\pi$-semibasic, we may write
\begin{align*}
\lambda_i & = T_{ij}\omega^j & \mu & = U_i\omega^i
\end{align*}
for some matrix-valued function $T = (T_{ij}) \colon F_{\text{SU}(3)} \to \text{Mat}_{6 \times 6}(\mathbb{R})$ and vector-valued function $U = (U_i) \colon F_{\text{SU}(3)} \to \mathbb{R}^6$.  The $1$-forms $\lambda, \mu$, and hence the functions $T_{ij}$ and $U_i$, encode the torsion of the $\text{SU}(3)$-structure.  In this notation, the first structure equation reads
\begin{equation}
d\omega_i = -(\gamma_{ij} + \epsilon_{ijk}\lambda_k + \Omega_{ij}\mu) \wedge \omega_j. \label{eq:SU3-FirstStrEqn}
\end{equation}
\begin{rmk}
	The reader may wonder how the functions $T_{ij}, U_i$ on $F_{\text{SU}(3)}$ are related to the forms $\tau_0, \widehat{\tau}_0, \ldots, \tau_4, \tau_5$ on $M$.  In \cite{MR2287296}, the authors derive expressions for the pullbacks of the torsion forms in terms of $T_{ij}, U_i$.  That is, they derive
	\begin{align*}
	\pi^*(\tau_0) & = -\textstyle \frac{1}{3} \Omega_{ij}T_{ij}  & \pi^*(\tau_4) & = \epsilon_{ijk}T_{ij}\,\omega^k \\
	\pi^*(\widehat{\tau}_0) & = \textstyle \frac{1}{3}T_{ii} &  \pi^*(\tau_5) & = \epsilon_{ijk}T_{ij}\omega^k + 3\Omega_{ik}U_i\,\omega^k
	\end{align*}
	along with similar (more complicated) formulas for $\pi^*(\tau_2), \pi^*(\widehat{\tau}_2), \pi^*(\tau_3)$.  In the next section, we will exhibit a sort of inverse to this, expressing the $T_{ij}, U_i$ in terms of the refined torsion forms $\pi^*(\tau_0), \pi^*(\widehat{\tau}_0), \ldots, \pi^*( (\tau_5)_{\mathsf{T}}), \pi^*( (\tau_5)_{\mathsf{N}})$.
\end{rmk}

\subsubsection{Decomposition of the Torsion Functions}\label{sssect:sLagTorsDecomp}

\indent \indent For our computations in \S\ref{ssect:sLagMC}, we will need to express the torsion functions $T_{ij}$ and $U_i$ in terms of the functions $A_p, B_\delta, \ldots, N_p$.  To this end, we will continue to work on the total space of the $\text{SU}(3)$-coframe bundle $\pi \colon F_{\text{SU}(3)} \to M$, pulling back all of the quantities defined on $M$ to $F_{\text{SU}(3)}$.  Following convention, we systematically omit $\pi^*$ from the notation, so that (for example) $\pi^*(\tau_0)$ will simply be denoted $\tau_0$, etc.  Note, however, that $\pi^*(e^j) = \omega_j$. \\

\indent To begin, recall that the torsion forms $\tau_0, \tau_1, \tau_2, \tau_3$ satisfy
\begin{align*}
d\Omega & = 3\tau_0\,\text{Re}(\Upsilon) + 3\widehat{\tau}_0\,\text{Im}(\Upsilon) + \tau_3 + \tau_4 \wedge \Omega, \\
d\,\text{Re}(\Upsilon) & = 2\widehat{\tau}_0\,\Omega^2 + \tau_5 \wedge \text{Re}(\Upsilon) + \tau_2 \wedge \Omega, \\
d\,\text{Im}(\Upsilon) & = -2\tau_0\,\Omega^2 - J\tau_5 \wedge \text{Re}(\Upsilon) + \widehat{\tau}_2 \wedge \Omega.
\end{align*}
Into the left-hand sides, we substitute (\ref{eq:OmegaUpsilon}) and use the first structure equation (\ref{eq:SU3-FirstStrEqn}) to obtain
\begin{align*}
\widehat{\epsilon}_{\ell j k}T_{\ell i}\, \omega^{i j k} & = 3\tau_0\,\text{Re}(\Upsilon) + 3\widehat{\tau}_0\,\text{Im}(\Upsilon) + \tau_3 + \tau_4 \wedge \Omega, \\
-\textstyle \frac{1}{2} \left( (\Omega_{k m}\Omega_{\ell j} - \Omega_{kj}\Omega_{\ell m})\,T_{mi} + \widehat{\epsilon}_{jk\ell} U_i \right) \omega^{ijk\ell} & = 2\widehat{\tau}_0\,\Omega^2 + \tau_5 \wedge \text{Re}(\Upsilon) + \tau_2 \wedge \Omega, \\
-\textstyle \frac{1}{2}\left(2\Omega_{k\ell}T_{ji} - \epsilon_{jk\ell}U_i  \right)\omega^{ijk\ell}  & = -2\tau_0\,\Omega^2 - J\tau_5 \wedge \text{Re}(\Upsilon) + \widehat{\tau}_2 \wedge \Omega.
\end{align*}
Into the right-hand sides, we again substitute (\ref{eq:OmegaUpsilon}), as well as the expansions (\ref{eq:SU3-TorsionSplit-1}) and (\ref{eq:SU3-TorsionSplit-2}).

Upon equating coefficients, we obtain a system of $56 = \binom{7}{4} + \binom{7}{5}$ linear equations relating the $42 = 6^2 + 6$ functions $T_{ij}, U_j$ on the left side to the $42 = \dim(H^{0,2}(\mathfrak{su}(3)))$ functions $\tau_0, \widehat{\tau}_0, A_p, B_\delta, \ldots, M_p, N_p$ on the right side.  One can then use a computer algebra system (we have used M\textsc{aple}) to solve this linear system for the $T_{ij}$ and $U_i$.

We now exhibit the result, taking advantage of the $\text{SO}(3)$-irreducible splitting
\begin{align*}
\text{Mat}_{6\times 6}(\mathbb{R}) \cong V^* \otimes V^* & \cong (\mathsf{T} \otimes \mathsf{T}) \oplus 2(\mathsf{T} \otimes \mathsf{N}) \oplus (\mathsf{N} \otimes \mathsf{N}) \\
& \cong \left( \Lambda^2(\mathsf{T})  \oplus \text{Sym}^2_0(\mathsf{T}) \oplus \mathbb{R} \right)  \oplus 2\!\left( \mathbb{R} \oplus (\mathsf{T} \otimes \mathsf{N})_1 \oplus (\mathsf{T} \otimes \mathsf{N})_2  \right) \\
& \ \ \ \oplus \left(\Lambda^2(\mathsf{N}) \oplus \text{Sym}^2_0(\mathsf{N}) \oplus \mathbb{R} \right)
\end{align*}
to highlight the structure of the solution.  We have
\begin{subequations} \label{eq:SU3-TorSol1}
\begin{align}
\frac{1}{2} \begin{bmatrix}
0 & T_{12} - T_{21} & T_{13} - T_{31} \\
T_{21} - T_{12} & 0 & T_{23} - T_{32}  \\
T_{31} - T_{13} & T_{32} - T_{23} & 0
\end{bmatrix} & =  \begin{bmatrix}
0 & - C_3 + 3G_3 & C_2 - 3G_2 \\
C_3 - 3G_3  & 0 & -C_1 + 3G_1 \\
-C_2 + 3G_2 & C_1 - 3G_1 & 0
\end{bmatrix} \\
\notag \\
\frac{1}{2}\begin{bmatrix}
2T_{11} & T_{12} + T_{21} & T_{13} + T_{31}  \\
T_{21} + T_{12} & 2T_{22} & T_{23} + T_{32}  \\
T_{31} + T_{13} & T_{32} + T_{23} & 2T_{33}
\end{bmatrix} & = \begin{bmatrix}
B_4 - F_4 & B_3 - F_3 & B_2 - F_2 \\
B_3 - F_3 & -B_4 + B_5 + F_4 - F_5 & B_1 - F_1 \\
B_2 - F_2 & B_1 - F_1 & -B_5 + F_5
\end{bmatrix} \notag \\
& \ \ \ \ \ + \left( \frac{1}{2}\widehat{\tau}_0 - 2F_0 \right)\text{Id}_3
\end{align}
\end{subequations}
corresponding to $\mathsf{T} \otimes \mathsf{T} \cong \Lambda^2(\mathsf{T})  \oplus \text{Sym}^2_0(\mathsf{T}) \oplus \mathbb{R}$ and
\begin{subequations}
\begin{align} \label{eq:SU3-TorSol2}
\frac{1}{2}\begin{bmatrix}
T_{14} - T_{41} & T_{24} - T_{42} & T_{34} - T_{43} \\
T_{15} - T_{51} & T_{25} - T_{52} & T_{35} - T_{53} \\
T_{16} - T_{61} & T_{26} - T_{62} & T_{36} - T_{63}
\end{bmatrix} & =
\begin{bmatrix}
D_4 & D_3 - 3J_3 & D_2 + 3J_2 \\
D_3 + 3J_3 & -D_4 + D_5 & D_1 - 3J_1 \\
D_2 - 3J_2 & D_1 + 3J_1 & -D_5
\end{bmatrix} - \frac{1}{2}\tau_0\,\text{Id}_3 \\
%\end{align}
%and
%\begin{align} 
\frac{1}{2}\begin{bmatrix}
T_{14} + T_{41} & T_{24} + T_{42} & T_{34} + T_{43} \\
T_{15} + T_{51} & T_{25} + T_{52} & T_{35} + T_{53} \\
T_{16} + T_{61} & T_{26} + T_{62} & T_{36} + T_{63}
\end{bmatrix} & =
\begin{bmatrix}
E_4 & A_3 + E_3 & -A_2 + E_2 \\
-A_3 + E_3 & -E_4 + E_5 & A_1 + E_1 \\
A_2 + E_2 & -A_1 + E_1 & -E_5
\end{bmatrix} + 2E_0\,\text{Id}_3
\end{align}
\end{subequations}
corresponding to $\mathsf{T} \otimes \mathsf{N} \cong \mathbb{R} \oplus (\mathsf{T} \otimes \mathsf{N})_1 \oplus (\mathsf{T} \otimes \mathsf{N})_2$, and
\begin{subequations} \label{eq:SU3-TorSol3}
\begin{align}
\frac{1}{2} \begin{bmatrix}
0 & T_{45} - T_{54} & T_{46} - T_{64} \\
T_{54} - T_{45} & 0 & T_{56} - T_{65} \\
T_{64} - T_{46} & T_{65} - T_{56} & 0
\end{bmatrix} & = \begin{bmatrix}
0 & -(C_3 + 3G_3) & C_2 + 3G_2 \\
C_3 + 3G_3 & 0 & -(C_1 + 3G_1) \\
-(C_2 + 3G_2) & C_1 + 3G_1 & 0
\end{bmatrix} \\
\notag \\
\frac{1}{2} \begin{bmatrix}
2T_{44} & T_{45} + T_{54} & T_{46} + T_{64} \\
T_{54} + T_{45} & 2T_{55} & T_{56} + T_{65} \\
T_{64} + T_{46} & T_{65} & 2T_{66}
\end{bmatrix} & = \begin{bmatrix}
B_4 + F_4 & B_3 + F_3 & B_2 + F_2 \\
B_3 + F_3 & -B_4 + B_5 - F_4 + F_5 & B_1 + F_1 \\
B_2 + F_2 & B_1 + F_1 & -B_5 - F_5
\end{bmatrix} \notag \\
& \ \ \ \ \ + \left( \frac{1}{2}\widehat{\tau}_0 + 2F_0 \right)\text{Id}_3
\end{align}
\end{subequations}
corresponding to $\mathsf{N} \otimes \mathsf{N} \cong \Lambda^2(\mathsf{N})  \oplus \text{Sym}^2_0(\mathsf{N}) \oplus \mathbb{R}$.  We also have
\begin{align} \label{eq:SU3-TorSol4}
\begin{bmatrix} U_1 \\ U_2 \\ U_3 \end{bmatrix} & = \begin{bmatrix} -4J_1 + N_1 \\ -4J_2 + N_2 \\ -4J_3 + N_3 \end{bmatrix} & \begin{bmatrix} U_4 \\ U_5 \\ U_6 \end{bmatrix} & = \begin{bmatrix} 4G_1 - M_1 \\ 4G_2 - M_2 \\ 4G_3 - M_3 \end{bmatrix}\!.
\end{align}

%%%%% Section 2.4 %%%%%

\subsection{Mean Curvature of Special Lagrangian $3$-Folds}\label{ssect:sLagMC}

\indent \indent In this section, we derive a formula (Theorem \ref{thm:sLagMC}) for the mean curvature of a special Lagrangian $3$-fold in an arbitrary $6$-manifold $(M, \Omega, \Upsilon)$ with $\text{SU}(3)$-structure $(\Omega, \Upsilon)$.  In the process, we observe a necessary condition (Theorem \ref{thm:sLagObs}) for the local existence of special Lagrangian $3$-folds.

We continue to let $\pi \colon F_{\text{SU}(3)} \to M$ denote the $\text{SU}(3)$-coframe bundle of $M$, and $\omega = (\omega_{\mathsf{T}}, \omega_{\mathsf{N}}) \in \Omega^1(F_{\text{SU}(3)}; \mathsf{T}^\sharp \oplus \mathsf{N}^\sharp)$ denote the tautological $1$-form.  As above, $\gamma = (\gamma_{ij}) \in \Omega^1(F_{\text{SU}(3)}; \mathfrak{su}(3))$ denotes the natural connection $1$-form, while $\lambda = (\lambda_{ij}) \in \Omega^1(F_{\text{SU}(3)}; \mathbb{R}^6)$ and $\mu \in \Omega^1(F_{\text{SU}(3)}; \mathbb{R})$ are $\pi$-semibasic $1$-forms encoding the torsion of $(\Omega, \Upsilon)$. \\

\indent Fix a phase $\theta \in [0,2\pi)$ once and for all, fix
$$t := \frac{\theta}{3},$$
and define $1$-forms $\eta, \xi \in \Omega^1(F_{\text{SU}(3)}; \mathbb{R}^6)$ via
\begin{align*}
\eta & = \text{Re}(e^{-it}(\omega_{\mathsf{T}} + i\omega_{\mathsf{N}})) = \cos(t)\,\omega_{\mathsf{T}} + \sin(t)\,\omega_{\mathsf{N}} \\
\xi & = \text{Im}(e^{-it}(\omega_{\mathsf{T}} + i\omega_{\mathsf{N}})) = -\sin(t)\,\omega_{\mathsf{T}} + \cos(t)\,\omega_{\mathsf{N}}.
\end{align*}
%%%%% PHASE EDIT MADE
Let $f \colon \Sigma^3 \to M^6$ denote an immersion of a phase $\theta$ special Lagrangian $3$-fold into $M$, and let $f^*(F_{\text{SU}(3)}) \to \Sigma$ denote the pullback bundle.  Let $B \subset f^*(F_{\text{SU}(3)})$ denote the subbundle of coframes adapted to $\Sigma$, i.e., the subbundle whose fiber over $x \in \Sigma$ is
\begin{align*}
B|_x & = \{u \in f^*(F_{\text{SU}(3)})|_x \colon u(T_x\Sigma) = e^{it} \cdot \mathsf{T}^\sharp \} \\
& = \{u \in f^*(F_{\text{SU}(3)})|_x \colon u(T_x\Sigma) = \text{span}(v_1(\theta), v_2(\theta), v_3(\theta)) \}
\end{align*}
%%%%% PHASE EDIT MADE
in the notation of (\ref{eq:vw-vectors}).  We recall (Proposition 2.3) that $\text{SU}(3)$ acts transitively on the set of special Lagrangian $3$-planes with stabilizer $\text{SO}(3)$, so $B \to \Sigma$ is a well-defined $\text{SO}(3)$-bundle.  Note that on $B$, we have
$$\xi = 0.$$
For the rest of \S\ref{ssect:sLagMC}, all of our calculations will be done on the subbundle $B \subset F_{\text{SU}(3)}$. \\

\indent We begin by expressing $\gamma$, $\lambda$, and $\mu$ as block matrices with respect to the splitting $T_xM \simeq \mathsf{T}^\sharp \oplus \mathsf{N}^\sharp$.  The $1$-form $\gamma \in \Omega^1(B; \mathfrak{su}(3))$ takes the block form
\begin{equation*}
\gamma = \begin{bmatrix} \alpha & \beta \\ -\beta & \alpha \end{bmatrix} = \left[ \begin{array}{c c c | c c c}
0 & \alpha_{12} & -\alpha_{13} & \beta_{11} & \beta_{12} & \beta_{13} \\
-\alpha_{12} & 0 & \alpha_{23} & \beta_{21} & \beta_{22} & \beta_{23} \\
\alpha_{13} & -\alpha_{23} & 0 & \beta_{31} & \beta_{32} & \beta_{33} \\ \hline
-\beta_{11} & -\beta_{12} & -\beta_{13} & 0 & \alpha_{12} & \alpha_{13} \\
-\beta_{21} & -\beta_{22} & -\beta_{23} & -\alpha_{12} & 0 & \alpha_{23} \\
-\beta_{31} & -\beta_{32} & -\beta_{33} & -\alpha_{13} & -\alpha_{23} & 0
\end{array}\right]
\end{equation*}
where $\alpha_{pq}, \beta_{pq} \in \Omega^1(B)$ are $1$-forms with $\beta_{pq} = \beta_{qp}$ and $\beta_{11} + \beta_{22} + \beta_{33} = 0$.  As in \S\ref{sssect:sLagTorsFuns}, the $1$-forms $\lambda \in \Omega^1(B; \mathbb{R}^6)$ and $\mu \in \Omega^1(B; \mathbb{R})$ break into blocks as
\begin{align*}
\lambda & = \begin{bmatrix} \lambda_{\mathsf{T}} & \lambda_{\mathsf{N}} \\ \lambda_{\mathsf{N}} & \lambda_{\mathsf{T}} \end{bmatrix} = \left[ \begin{array}{c c c | c c c}
0 & \lambda_3 & -\lambda_2 & 0 & -\lambda_6 & \lambda_5 \\
-\lambda_3 & 0 & \lambda_1 & \lambda_6 & 0 & -\lambda_4 \\
\lambda_2 & -\lambda_1 & 0 & -\lambda_5 & \lambda_4 & 0 \\ \hline
0 & -\lambda_6 & \lambda_5 & 0 & -\lambda_3 & \lambda_2 \\
\lambda_6 & 0 & -\lambda_4 & \lambda_3 & 0 & -\lambda_1 \\
-\lambda_5 & \lambda_4 & 0 & -\lambda_2 & \lambda_1 & 0
\end{array}\right] & \mu & = \begin{bmatrix} 0 & \mu\,\text{Id}_3 \\ -\mu\,\text{Id}_3 & 0 \end{bmatrix}\!.
\end{align*}

\indent Next, we adapt our matrix-valued forms to the geometry at hand, which is that of a splitting $T_xM \simeq T_x\Sigma \oplus (T_x\Sigma)^\perp$.  To this end, recall that the change-of-phase action on $V \simeq \mathbb{C}^3$ is the $\mathbb{S}^1$-action given by $e^{i\vartheta} \cdot (z_1, z_2, z_3) = (e^{i\vartheta}z_1, e^{i\vartheta}z_2, e^{i\vartheta}z_3)$.  Regarding this $\mathbb{S}^1$ as a subgroup of $\text{U}(3) \leq \text{SO}(6)$, we consider the induced $\text{Ad}(\mathbb{S}^1)$-action on $\mathfrak{so}(6)$, given explicitly in block form as
$$\text{Ad}_\vartheta \begin{bmatrix}
A & B \\ C & D
\end{bmatrix} = \begin{bmatrix}
\cos(\vartheta)\,\text{Id}_3 & \sin(\vartheta)\,\text{Id}_3 \\
-\sin(\vartheta)\,\text{Id}_3 & \cos(\vartheta)\,\text{Id}_3 \end{bmatrix}
\begin{bmatrix}
A & B \\ C & D
\end{bmatrix}
\begin{bmatrix}
\cos(\vartheta)\,\text{Id}_3 & -\sin(\vartheta)\,\text{Id}_3 \\
\sin(\vartheta)\,\text{Id}_3 & \cos(\vartheta)\,\text{Id}_3 \end{bmatrix}\!.$$
%%%%% PHASE EDIT MADE
Viewing $\mathfrak{so}(6) = \mathfrak{su}(3) \oplus \mathbb{R} \oplus \mathbb{R}^6$, note that our $\text{Ad}(\mathbb{S}^1)$-action is trivial on the $\mathfrak{su}(3)$- and $\mathbb{R}$-summands, and thus
$$\text{Ad}_t \gamma = \gamma \ \ \ \ \ \ \ \ \ \text{Ad}_t \mu = \mu.$$
However, the $\text{Ad}(\mathbb{S}^1)$-action is non-trivial on the $\mathbb{R}^6$-summand.  We therefore set $\widetilde{\lambda} = \text{Ad}_t \lambda$, writing
%%%%% PHASE EDIT MADE
\begin{align*}
\widetilde{\lambda} & = \begin{bmatrix} \widetilde{\lambda}_{\mathsf{T}} & \widetilde{\lambda}_{\mathsf{N}} \\ \widetilde{\lambda}_{\mathsf{N}} & \widetilde{\lambda}_{\mathsf{T}} \end{bmatrix} = \left[ \begin{array}{c c c | c c c}
0 & \widetilde{\lambda}_3 & -\widetilde{\lambda}_2 & 0 & -\widetilde{\lambda}_6 & \widetilde{\lambda}_5 \\
-\widetilde{\lambda}_3 & 0 & \widetilde{\lambda}_1 & \widetilde{\lambda}_6 & 0 & -\widetilde{\lambda}_4 \\
\widetilde{\lambda}_2 & -\widetilde{\lambda}_1 & 0 & -\widetilde{\lambda}_5 & \widetilde{\lambda}_4 & 0 \\ \hline
0 & -\widetilde{\lambda}_6 & \widetilde{\lambda}_5 & 0 & -\widetilde{\lambda}_3 & \widetilde{\lambda}_2 \\
\widetilde{\lambda}_6 & 0 & -\widetilde{\lambda}_4 & \widetilde{\lambda}_3 & 0 & -\widetilde{\lambda}_1 \\
-\widetilde{\lambda}_5 & \widetilde{\lambda}_4 & 0 & -\widetilde{\lambda}_2 & \widetilde{\lambda}_1 & 0
\end{array}\right]\!.
\end{align*}
Explicitly, we have formulas
\begin{align*}
\widetilde{\lambda}_1 & =  \cos(2t) \lambda_1 - \sin(2t) \lambda_4 &  \widetilde{\lambda}_4 & =  \sin(2t) \lambda_1 +  \cos(2t) \lambda_4 \\
\widetilde{\lambda}_2 & =  \cos(2t) \lambda_2 - \sin(2t) \lambda_5 & \widetilde{\lambda}_5 & =  \sin(2t)\lambda_2 +  \cos(2t)\lambda_5 \\
\widetilde{\lambda}_3 & =  \cos(2t) \lambda_3 - \sin(2t) \lambda_6 & \widetilde{\lambda}_6 & = \sin(2t) \lambda_3 + \cos(2t) \lambda_6.
\end{align*} \\
%%%%% PHASE EDIT MADE
\indent We may now apply these $\mathbb{S}^1$-actions to the structure equation (\ref{eq:SU3-FirstStrEqn}) on $F_{\text{SU}(3)}$.  Using that $\xi = 0$ on $B \subset F_{\text{SU}(3)}$, we deduce the first structure equation on $B$:
$$d\begin{pmatrix} \eta \\ 0 \end{pmatrix} = -\left( \begin{bmatrix} \alpha & \beta \\ -\beta & \alpha \end{bmatrix} + \begin{bmatrix} \widetilde{\lambda}_{\mathsf{T}} & \widetilde{\lambda}_{\mathsf{N}} \\ \widetilde{\lambda}_{\mathsf{N}} & \widetilde{\lambda}_{\mathsf{T}} \end{bmatrix}  + \begin{bmatrix} 0 & \mu\,\text{Id}_3 \\ -\mu\,\text{Id}_3 & 0 \end{bmatrix} \right) \wedge \begin{pmatrix} \eta \\ 0 \end{pmatrix}\!.$$
In particular, the second line gives
$$\beta \wedge \eta = (\widetilde{\lambda}_{\mathsf{N}} - \mu\,\text{Id}_3) \wedge \eta$$
or in detail,
\begin{align} \label{eq:SU3-CondensedStrEqn}
& \begin{bmatrix}
\beta_{11} & \beta_{12} & \beta_{13} \\
\beta_{21} & \beta_{22} & \beta_{23} \\
\beta_{31} & \beta_{32} & \beta_{33} 
\end{bmatrix} \wedge \begin{bmatrix} \eta_1 \\ \eta_2 \\ \eta_3 \end{bmatrix} = \begin{bmatrix}
-\mu & -\widetilde{\lambda}_6 & \widetilde{\lambda}_5 \\
\widetilde{\lambda}_6 & -\mu & -\widetilde{\lambda}_4 \\
-\widetilde{\lambda}_5 & \widetilde{\lambda}_4 & -\mu
\end{bmatrix}
\wedge \begin{bmatrix} \eta_1 \\ \eta_2 \\ \eta_3 \end{bmatrix}
\end{align}
Note that on $B$, the $1$-forms $\beta_{pq}$, $\lambda_j$, and $\mu$ are semibasic, and we write
\begin{align*}
\beta_{pq} & = S_{pqr}\eta_r & \lambda_i & = T_{ij}\omega^j & \mu & = U_i\omega^i
\end{align*}
for some function $S = (S_{pqr}) \colon B \to \text{Sym}^2_0(\mathbb{R}^3) \otimes \mathbb{R}^3$. \\

\indent Now, the $15$ functions $S_{pqr}$ and the $42$ functions $(T_{ij}, U_i)$ are not independent: the equation (\ref{eq:SU3-CondensedStrEqn}) shows that they satisfy $3 \binom{3}{2} = 9$ linear relations.  Explicitly, the first row of (\ref{eq:SU3-CondensedStrEqn}) gives
\begin{align*}
\begin{bmatrix}
S_{123} - S_{132} \\
S_{113} - S_{131} \\
S_{121} - S_{112}
 \end{bmatrix} & = \frac{1}{2}\begin{bmatrix}
T_{22} + T_{33} + T_{55} + T_{66} & T_{21} + T_{54} &  T_{31} + T_{64} \\
T_{25} - T_{52} + T_{36} - T_{63} & T_{24} - T_{51} &  T_{34} - T_{61} \\
T_{22} + T_{33} - T_{55} - T_{66} & T_{21} - T_{54} + 2U_6 & T_{31} - T_{64} - 2U_5 \\
-T_{25} - T_{52} - T_{36} - T_{63} & -T_{24} - T_{51} - 2U_3 & -T_{34} - T_{61} + 2U_2
\end{bmatrix}^T
\begin{bmatrix}
-\sin(3t) \\
\cos(3t) \\
-\sin(t) \\
\cos(t)
\end{bmatrix}
\end{align*}
while the second row gives
\begin{align*}
\begin{bmatrix}
S_{232} - S_{223} \\
S_{231} - S_{123} \\
S_{122} - S_{221}
 \end{bmatrix} & = \frac{1}{2}\begin{bmatrix}
T_{12} + T_{45} & T_{11} + T_{33} + T_{44} + T_{66} &  -T_{32} - T_{65} \\
T_{15} - T_{42} & T_{14} - T_{41} + T_{36} - T_{63} &  -T_{35} + T_{62} \\
T_{12} - T_{45} - 2U_6 & T_{11} + T_{33} - T_{44} - T_{66} &  -T_{32} + T_{65} - 2U_4 \\
-T_{15} - T_{42} + 2U_3 & -T_{14} - T_{41} - T_{36} - T_{63} &  T_{35} + T_{62} + 2U_1
\end{bmatrix}^T
\begin{bmatrix}
-\sin(3t) \\
\cos(3t) \\
-\sin(t) \\
\cos(t)
\end{bmatrix}
\end{align*}
and the third row gives
\begin{align*}
\begin{bmatrix}
S_{233} - S_{332} \\
S_{133} - S_{331} \\
S_{132} - S_{231}
 \end{bmatrix} & = \frac{1}{2}\begin{bmatrix}
-T_{13} - T_{46} & T_{23} + T_{56} & T_{11} + T_{22} + T_{44} + T_{55} \\
-T_{16} + T_{43} & T_{26} - T_{53} &  T_{14} - T_{41} + T_{25} - T_{52} \\
-T_{13} + T_{46} - 2U_5 & T_{23} - T_{56} - 2U_4 & T_{11} + T_{22} - T_{44} - T_{55} \\
T_{16} + T_{43} + 2U_2 & -T_{26} - T_{53} + 2U_1 & -T_{14} - T_{41} - T_{25} - T_{52}
\end{bmatrix}^T
\begin{bmatrix}
-\sin(3t) \\
\cos(3t) \\
-\sin(t) \\
\cos(t)
\end{bmatrix}
\end{align*} \\
%%%%% PHASE EDIT MADE
\indent We make two observations on this system of linear equations.  First, we notice that it implies
\begin{align*}
0 & = (S_{123} - S_{132}) + (S_{231} - S_{123}) + (S_{132} - S_{231}) \\
& = -(T_{11} + T_{22} + T_{33} + T_{44} + T_{55} + T_{66}) \sin(3t) + (T_{14} - T_{41} + T_{25} - T_{52} + T_{36} - T_{63}) \cos(3t)  \\
& \ \ \ \ - (T_{11} + T_{22} + T_{33} - T_{44} - T_{55} - T_{66}) \sin(t) - (T_{14} + T_{41} + T_{25} + T_{52} + T_{36} + T_{63}) \cos(t).
\end{align*}
%%%%% PHASE EDIT MADE
Using (\ref{eq:SU3-TorSol1})-(\ref{eq:SU3-TorSol3}), we obtain:
$$0 = \widehat{\tau}_0 \sin(3t) + \tau_0 \cos(3t) - 4F_0 \sin(t) + 4E_0 \cos(t).$$
%%%%% PHASE EDIT MADE
Thus from (\ref{eq:SU3-Isom}), we deduce:
\begin{thm}\label{thm:sLagObs}
	If a special Lagrangian $3$-fold $\Sigma$ of phase $\theta$ exists in $M$, then the following relation holds at points of $\Sigma$:
	\begin{equation}
	\widehat{\tau}_0 \sin(\theta) + \tau_0 \cos(\theta) = \textstyle \frac{\sqrt{3}}{6} \left( \sin(\frac{\theta}{3})  [(\tau_3)_0'']^\ddag - \cos(\frac{\theta}{3}) [(\tau_3)_0']^\dagger \displaystyle \right)\!.
	\label{eq:SL-Obs}
	\end{equation}
	\noindent In particular, if $\tau_3 = 0$, then the phase of every special Lagrangian $3$-fold in $M$ satisfies the relation $\tan(\theta) = -\tau_0/\widehat{\tau}_0$.
\end{thm}
%%%%% PHASE EDIT MADE

% In particular, if the torsion of $M$ takes values in ....., then $M$ admits no special Lagrangian $3$-folds of any phase (even locally). \\
\begin{cor}\label{cor:sLagObs}
	Fix $x \in M$ and $\theta \in [0,2\pi)$.  If every phase $\theta$ special Lagrangian $3$-plane in $T_xM$ is tangent to a phase $\theta$ special Lagrangian $3$-fold, then $\tau_3|_x = 0$ and $\widehat{\tau}_0|_x \sin(\theta) = -\tau_0|_x \cos(\theta)$.
\end{cor}
%%%%% PHASE EDIT MADE
\begin{proof}[Proof of Corollary \ref{cor:sLagObs}]
	The hypotheses imply that equation (\ref{eq:SL-Obs}) holds for all phase $\theta$ special Lagrangian $3$-planes at $x \in M$. Thus, we get an $\text{SU}(3)$-invariant linear relation between $\tau_0,$ $\widehat{\tau}_0$, and $\tau_3.$ This implies that $\tau_3=0$ by Schur's Lemma. The statement $\widehat{\tau}_0|_x \sin(\theta) = -\tau_0|_x \cos(\theta)$ follows immediately.
\end{proof}
%%%%% PHASE EDIT MADE

Second, after using $S_{11r} + S_{22r} + S_{33r} = 0$ for each $r = 1,2,3$, we observe that:
\begin{subequations} \label{eq:STRel}
\begin{align}
& 2(S_{111} + S_{122} + S_{133}) =  (T_{26} + T_{62} - T_{35} - T_{53}) \cos(3t) - ( T_{23} - T_{32} + T_{56} - T_{65}) \sin(3t) \notag \\
& \ \ \ \ \ \ \ \  - ( T_{23} - T_{32} - T_{56} + T_{65} - 4U_4) \sin(t) + ( -T_{26} + T_{62} + T_{35} - T_{53} + 4U_1) \cos(t) \\
\notag \\
& 2(S_{121} + S_{222} + S_{233})  = (-T_{16} - T_{61} + T_{34} + T_{43}) \cos(3t) - (-T_{13} + T_{31} - T_{46} + T_{64}) \sin(3t) \notag \\
& \ \ \ \ \ \ \ \ - (-T_{13} + T_{31} + T_{46} - T_{64} - 4U_5) \sin(t) + (T_{16} - T_{61} - T_{34} + T_{43} + 4U_2 ) \cos(t) \\
\notag \\
& 2(S_{131} + S_{232} + S_{333}) = (T_{15} + T_{51} - T_{24} - T_{42}) \cos(3t) - ( T_{12} - T_{21} + T_{45} - T_{54} ) \sin(3t) \notag \\
& \ \ \ \ \ \ \ \  - ( T_{12} - T_{21} - T_{45} + T_{54} - 4U_6 ) \sin(t) + ( -T_{15} +T_{51} + T_{24} - T_{42} + 4U_3 ) \cos(t).
\end{align}
\end{subequations}
%%%%% PHASE EDIT MADE
We are now ready to compute the mean curvature of a phase $\theta$ special Lagrangian $3$-fold.
\begin{thm}\label{thm:sLagMC}
	Let $\Sigma \subset M$ be a special Lagrangian $3$-fold immersed in a $6$-manifold $M$ equipped with an $\text{SU}(3)$-structure.  Then the mean curvature vector $H$ of $\Sigma$ is given by
	$$\textstyle H = -\frac{1}{\sqrt{2}}\cos(\theta)\,[ (\tau_2)_1 ]^\natural - \frac{1}{\sqrt{2}}\sin(\theta)\,[ (\widehat{\tau}_2)_1 ]^\natural + \textstyle \sin(\frac{\theta}{3})\,[(\tau_5)_{\mathsf{T}}]^\S - \cos(\frac{\theta}{3})\,[J(\tau_5)_{\mathsf{N}}]^\S.$$
%%%%% PHASE EDIT MADE	
	In particular, the largest torsion class of $\text{SU}(3)$-structures $(\Omega, \Upsilon)$ for which every special Lagrangian $3$-fold (of every phase) is minimal is $X_0^+ \oplus X_0^- \oplus X_3 \oplus X_4$.
\end{thm}
\begin{proof}
	The mean curvature vector may be computed as follows:
	\begin{align} \label{eq:SU3-H}
	\begin{bmatrix} H_1 \\ H_2 \\ H_3 \end{bmatrix} \eta^{123} %& = \begin{bmatrix}
	%\psi_{41} & \psi_{42} & \psi_{43} \\
	%\psi_{51} & \psi_{52} & \psi_{53} \\
	%\psi_{61} & \psi_{62} & \psi_{63}
	%\end{bmatrix} \wedge \begin{bmatrix} \eta^{23} \\ \eta^{31} \\ \eta^{12} 
	%\end{bmatrix}  \\
	& = \begin{bmatrix}
	-\beta_{11} & -\beta_{12} & -\beta_{13} \\
	-\beta_{21} & -\beta_{22} & -\beta_{23} \\
	-\beta_{31} & -\beta_{32} & -\beta_{33} 
	\end{bmatrix} \wedge \begin{bmatrix} \eta^{23} \\ \eta^{31} \\ \eta^{12} 
	\end{bmatrix} + \begin{bmatrix}
	-\mu & -\widetilde{\lambda}_6 & \widetilde{\lambda}_5 \\
	\widetilde{\lambda}_6 & -\mu & -\widetilde{\lambda}_4 \\
	-\widetilde{\lambda}_5 & \widetilde{\lambda}_4 & -\mu
	\end{bmatrix}  \wedge \begin{bmatrix} \eta^{23} \\ \eta^{31} \\ \eta^{12} 
	\end{bmatrix}
	\end{align}
	To evaluate the first term of (\ref{eq:SU3-H}), we substitute $\beta_{pq} = S_{pqr}\eta_r$, followed by (\ref{eq:STRel}), and finally (\ref{eq:SU3-TorSol1})-(\ref{eq:SU3-TorSol4}), to obtain:
	\begin{align*}
	& \begin{bmatrix}
	-\beta_{11} & -\beta_{12} & -\beta_{13} \\
	-\beta_{21} & -\beta_{22} & -\beta_{23} \\
	-\beta_{31} & -\beta_{32} & -\beta_{33} 
	\end{bmatrix} \wedge \begin{bmatrix} \eta^{23} \\ \eta^{31} \\ \eta^{12} 
	\end{bmatrix} = -\begin{bmatrix}
	S_{111} + S_{122} + S_{133} \\
	S_{121} + S_{222} + S_{233} \\
	S_{131} + S_{232} + S_{333} \\
	\end{bmatrix} \eta_{123} \\
	& = 2\begin{bmatrix}
	-A_1 \cos(3t) - C_1\sin(3t) + (M_1 - G_1) \sin(t) + (J_1 - N_1) \cos(t)  \\
	-A_2 \cos(3t) - C_2 \sin(3t) + (M_2 - G_2) \sin(t) + (J_2 - N_2 ) \cos(t) \\
	-A_3 \cos(3t) - C_3 \sin(3t) + (M_3 - G_3) \sin(t) + (J_3 - N_3 ) \cos(t)
	\end{bmatrix} \eta_{123}
	\end{align*}
%%%%% PHASE EDIT MADE	
	Similarly, to evaluate the second term of (\ref{eq:SU3-H}), we substitute $\lambda_i = T_{ij}\omega^j$ and $\mu = U_i\omega^i$ followed by (\ref{eq:SU3-TorSol1})-(\ref{eq:SU3-TorSol4}) to obtain:
	\begin{align*}
	& \begin{bmatrix}
	-\mu & -\widetilde{\lambda}_6 & \widetilde{\lambda}_5 \\
	\widetilde{\lambda}_6 & -\mu & -\widetilde{\lambda}_4 \\
	-\widetilde{\lambda}_5 & \widetilde{\lambda}_4 & -\mu
	\end{bmatrix} \wedge \begin{bmatrix} \eta^{23} \\ \eta^{31} \\ \eta^{12} 
	\end{bmatrix} \\
	& = 2\begin{bmatrix}
	-A_1 \cos(3t) - C_1 \sin(3t)  + (G_1 + \frac{1}{2}M_1) \sin(t) - (J_1 + \frac{1}{2}N_1) \cos(t) \\
	-A_2 \cos(3t) - C_2 \sin(3t) + (G_2 + \frac{1}{2}M_2) \sin(t) - (J_2 + \frac{1}{2}N_2) \cos(t) \\
	-A_3 \cos(3t) - C_3 \sin(3t) + (G_3 + \frac{1}{2}M_3) \sin(t) - (J_3 + \frac{1}{2}N_3) \cos(t)
	\end{bmatrix} \eta_{123}
	\end{align*}
%%%%% PHASE EDIT MADE
	We conclude that
	$$H_p = -4 A_p \cos(3t) - 4 C_p \sin(3t) + 3 M_p \sin(t) - 3N_p \cos(t),$$
%%%%% PHASE EDIT MADE
	and so (\ref{eq:SU3-Isom}) and $t = \theta/3$ yields
	$$\textstyle H = -\frac{1}{\sqrt{2}}\cos(\theta)\,[ (\tau_2)_1 ]^\natural - \frac{1}{\sqrt{2}}\sin(\theta)\,[ (\widehat{\tau}_2)_1 ]^\natural + \textstyle \sin(\frac{\theta}{3})\,[(\tau_5)_{\mathsf{T}}]^\S - \cos(\frac{\theta}{3})\,[J(\tau_5)_{\mathsf{N}}]^\S.$$
%%%%% PHASE EDIT MADE
	Thus, the largest torsion class of $\text{SU}(3)$-structures for which $H = 0$ for all phases is the one for which $\tau_2 = \widehat{\tau}_2 = \tau_5 = 0$, namely $X_0^+ \oplus X_0^- \oplus X_3 \oplus X_4$.
\end{proof}
\begin{rmk}
	In the following table, we summarize the results above for certain special classes of $\text{SU}(3)$-structures encountered in the literature.
	\begin{center}
		\begin{tabular}{ | c | c | c | c | }
			\hline
			Name & Torsion Class & Mean Curvature & Necessary Condition for \\ 
			&  & of Phase $\theta$ SLags & Local Existence of Phase $\theta$ \\     
			&  &  & SLag at a Point \\ \hline \hline    
			CY & $0$ & $0$ & $-$ \\ \hline
			NK 1 & $X_0^+$ & $0$ & $\tau_0 \cos(\theta) = 0$ \\ \hline
			NK 2 & $X_0^-$ & $0$ & $\widehat{\tau}_0 \sin(\theta) = 0$ \\ \hline
			GCY & $X_2^+ \oplus X_2^-$ & $-\frac{1}{\sqrt{2}}(\sin(\theta)\,[ (\widehat{\tau}_2)_1 ]^\natural $ & $-$ \\ 
			& & $ \ \ \ \ \ \ \ \ + \cos(\theta)\,[ (\tau_2)_1 ]^\natural)$   &  \\ \hline
			Half-Flat & $X_0^+ \oplus X_2^- \oplus X_3$ & $-\frac{1}{\sqrt{2}}\sin(\theta)\,[ (\widehat{\tau}_2)_1 ]^\natural$ & $\frac{\sqrt{3}}{6}\! \left( \sin(\frac{\theta}{3})  [(\tau_3)_0'']^\ddag - \cos(\frac{\theta}{3}) [(\tau_3)_0']^\dagger \right)$ \\ 
			&  &  & $= \tau_0 \cos(\theta)$ \\ \hline
			Symp Half-Flat & $X_2^-$ & $-\frac{1}{\sqrt{2}}\sin(\theta)\,[ (\widehat{\tau}_2)_1 ]^\natural$ & $-$ \\ \hline         
			Balanced & $X_3$ & $0$ & $\sin(\frac{\theta}{3})  [(\tau_3)_0'']^\ddag = \cos(\frac{\theta}{3}) [(\tau_3)_0']^\dagger$ \\ \hline        
			Class $X_4$        & $X_4$ & $0$ & $-$ \\ \hline            
		\end{tabular}
	\end{center}
%%% PHASE EDITS MADE	
	Here, we are using the shorthand
	\begin{align*}
	& \text{CY = Calabi-Yau} \\
	& \text{NK 1 = Nearly-K\"{a}hler with convention } d\Omega = 3\tau_0\,\text{Re}(\Upsilon) \text{ and } d\,\text{Im}(\Upsilon) = -2\tau_0\,\Omega^2 \\
	& \text{NK 2 = Nearly-K\"{a}hler with convention } d\Omega = 3\widehat{\tau}_0\,\text{Im}(\Upsilon) \text{ and } d\,\text{Re}(\Upsilon) = 2\widehat{\tau}_0\,\Omega^2 \\
	& \text{GCY = Generalized Calabi-Yau} \\
	& \text{Symp Half-Flat = Symplectic Half-Flat = Special Generalized Calabi-Yau (SGCY)}
	\end{align*}
	Both conventions for nearly-K\"{a}hler $6$-manifolds are found in the literature (contrast, say, \cite{MR3583352} with \cite{bryant2006geometry}).  Generalized Calabi-Yau structures are studied in, for example, \cite{de2006maslov} and \cite{MR2287296}. \\
	\indent Half-flat structures have been used by Hitchin \cite{hitchin2001stable} to construct $\text{G}_2$-manifolds via evolution equations.  Symplectic half-flat structures are studied in \cite{tomassini2008symplectic} and \cite{MR2287296}, the latter work referring to them as ``special generalized Calabi-Yau" structures. \\
	\indent Balanced $\text{SU}(3)$-structures on connected sums of copies of $\mathbb{S}^3 \times \mathbb{S}^3$ are constructed in \cite{fu2012balanced}.  Hypersurfaces in $6$-manifolds with balanced $\text{SU}(3)$-structures are studied in \cite{fernandez2009balanced}.  Nilmanifolds with $\text{SU}(3)$-structures of class $X_4$ are constructed in $\S$4.6 of \cite{schoemann2007almost}.
\end{rmk}

% \pagebreak

\bibliographystyle{plain}
\bibliography{MeanCurvRef}

\begin{thebibliography}{10}

\bibitem{BaMaExcept}
Gavin {Ball} and Jesse {Madnick}.
\newblock The mean curvature of first-order submanifolds in exceptional
  geometries with torsion.
\newblock {\em Ann. Global Anal. Geom.}, in press.

\bibitem{bar1993real}
Christian B{\"a}r.
\newblock Real killing spinors and holonomy.
\newblock {\em Communications in mathematical physics}, 154(3):509--521, 1993.

\bibitem{MR2287296}
Lucio Bedulli and Luigi Vezzoni.
\newblock The {R}icci tensor of {SU}(3)-manifolds.
\newblock {\em J. Geom. Phys.}, 57(4):1125--1146, 2007.

\bibitem{BrotD85}
Theodor Br\"{o}cker and Tammo tom Dieck.
\newblock {\em Representations of compact {L}ie groups}, volume~98 of {\em
  Graduate Texts in Mathematics}.
\newblock Springer-Verlag, New York, 1985.

\bibitem{bryant2000second}
Robert~L Bryant.
\newblock Second order families of special lagrangian 3-folds.
\newblock {\em arXiv preprint math.DG/0007128}, 2000.

\bibitem{bryant2006geometry}
Robert~L Bryant.
\newblock On the geometry of almost complex 6-manifolds.
\newblock {\em Asian Journal of Mathematics}, 10(3):561--605, 2006.

\bibitem{bryant2006so}
Robert~L. Bryant.
\newblock {SO}{(n)}-{I}nvariant {S}pecial {L}agrangian {S}ubmanifolds of
  {$C^{n+1}$} with {F}ixed {L}oci.
\newblock {\em Chinese Annals of Mathematics, Series B}, 27(1):95--112, 2006.

\bibitem{butscher2001regularizing}
Adrian Butscher.
\newblock Regularizing a singular special lagrangian variety.
\newblock {\em arXiv preprint math/0110053}, 2001.

\bibitem{de2006maslov}
Paolo De~Bartolomeis and Adriano Tomassini.
\newblock On the {M}aslov index of {L}agrangian submanifolds of generalized
  {C}alabi--{Y}au manifolds.
\newblock {\em International Journal of Mathematics}, 17(08):921--947, 2006.

\bibitem{fernandez2009balanced}
Marisa Fern{\'a}ndez, Adriano Tomassini, Luis Ugarte, and Raquel Villacampa.
\newblock Balanced {H}ermitian metrics from {SU}(2)-structures.
\newblock {\em Journal of Mathematical Physics}, 50(3):033507, 2009.

\bibitem{MR3583352}
Lorenzo Foscolo and Mark Haskins.
\newblock New {$G_2$}-holonomy cones and exotic nearly {K}{\"a}hler structures
  on {$S^6$} and {$S^3\times S^3$}.
\newblock {\em Ann. of Math. (2)}, 185(1):59--130, 2017.

\bibitem{fu2012balanced}
Jixiang Fu, Jun Li, and Shing-Tung Yau.
\newblock Balanced metrics on non-{K}{\"a}hler {C}alabi-{Y}au threefolds.
\newblock {\em Journal of Differential Geometry}, 90(1):81--129, 2012.

\bibitem{MR1062197}
Robert~B. Gardner.
\newblock {\em The method of equivalence and its applications}, volume~58 of
  {\em CBMS-NSF Regional Conference Series in Applied Mathematics}.
\newblock Society for Industrial and Applied Mathematics (SIAM), Philadelphia,
  PA, 1989.

\bibitem{grunewald1990six}
Ralf Grunewald.
\newblock Six-dimensional riemannian manifolds with a real killing spinor.
\newblock {\em Annals of Global Analysis and Geometry}, 8(1):43--59, 1990.

\bibitem{harvey1982calibrated}
Reese Harvey and H~Blaine Lawson.
\newblock Calibrated geometries.
\newblock {\em Acta Mathematica}, 148(1):47--157, 1982.

\bibitem{haskins2007special}
Mark Haskins and Nikolaos Kapouleas.
\newblock Special lagrangian cones with higher genus links.
\newblock {\em Inventiones mathematicae}, 167(2):223--294, 2007.

\bibitem{hitchin2001stable}
Nigel Hitchin.
\newblock Stable forms and special metrics.
\newblock {\em Contemporary mathematics}, 288:70--89, 2001.

\bibitem{hitchin1997moduli}
Nigel~J Hitchin.
\newblock The moduli space of special lagrangian submanifolds.
\newblock {\em Annali della Scuola Normale Superiore di Pisa-Classe di
  Scienze}, 25(3-4):503--515, 1997.

\bibitem{ionel2008cohomogeneity}
Marianty Ionel and Maung Min-Oo.
\newblock Cohomogeneity one special lagrangian 3-folds in the deformed and the
  resolved conifolds.
\newblock {\em Illinois Journal of Mathematics}, 52(3):839--865, 2008.

\bibitem{joyce20010}
Dominic Joyce.
\newblock {C}onstructing special {L}agrangian m-folds in {$C^m$} by evolving
  quadrics.
\newblock {\em Mathematische Annalen}, 320(4):757--797, 2001.

\bibitem{joyce2002ruled}
Dominic Joyce.
\newblock {R}uled {S}pecial {L}agrangian 3-folds in {$C^3$}.
\newblock {\em Proceedings of the London Mathematical Society}, 85(1):233--256,
  2002.

\bibitem{joyce2002special}
Dominic Joyce.
\newblock {S}pecial {L}agrangian {$m$}-folds in {$C^m$} with symmetries.
\newblock {\em Duke Mathematical Journal}, 115(1):1--51, 2002.

\bibitem{joyce2007riemannian}
Dominic~D Joyce.
\newblock {\em Riemannian holonomy groups and calibrated geometry}, volume~12.
\newblock Oxford University Press, 2007.

\bibitem{lawlor1989angle}
Gary Lawlor.
\newblock The angle criterion.
\newblock {\em Inventiones mathematicae}, 95(2):437--446, 1989.

\bibitem{LaMiSpin89}
H.~Blaine Lawson, Jr. and Marie-Louise Michelsohn.
\newblock {\em Spin geometry}, volume~38 of {\em Princeton Mathematical
  Series}.
\newblock Princeton University Press, Princeton, NJ, 1989.

\bibitem{lee2003connected}
Dan~A Lee.
\newblock Connected sums of special lagrangian submanifolds.
\newblock {\em arXiv preprint math/0303224}, 2003.

\bibitem{lee2003embedded}
Yng-Ing Lee.
\newblock Embedded special lagrangian submanifolds in calabi-yau manifolds.
\newblock {\em Communications in Analysis and Geometry}, 11(3):391--423, 2003.

\bibitem{li2005existence}
Jun Li and Shing-Tung Yau.
\newblock The existence of supersymmetric string theory with torsion.
\newblock {\em J. Diff. Geom}, 70(1):143--181, 2005.

\bibitem{mclean1998deformations}
Robert~C McLean.
\newblock Deformations of calibrated submanifolds.
\newblock {\em Communications in Analysis and Geometry}, 6(4):705--747, 1998.

\bibitem{pacini2013special}
Tommaso Pacini.
\newblock Special {L}agrangian conifolds, {II}: gluing constructions in
  {$C^m$}.
\newblock {\em Proceedings of the London Mathematical Society},
  107(2):225--266, 2013.

\bibitem{MR1004008}
Simon Salamon.
\newblock {\em Riemannian geometry and holonomy groups}, volume 201 of {\em
  Pitman Research Notes in Mathematics Series}.
\newblock Longman Scientific \& Technical, Harlow; copublished in the United
  States with John Wiley \& Sons, Inc., New York, 1989.

\bibitem{Salur00}
Sema Salur.
\newblock Deformations of special {L}agrangian submanifolds.
\newblock {\em Commun. Contemp. Math.}, 2(3):365--372, 2000.

\bibitem{schoemann2007almost}
Nils Schoemann.
\newblock Almost hermitian structures with parallel torsion.
\newblock {\em Journal of Geometry and Physics}, 57(11):2187--2212, 2007.

\bibitem{schulte2010half}
Fabian Schulte-Hengesbach.
\newblock {\em Half-flat structures on {L}ie groups}.
\newblock PhD thesis, Universit{\"a}t Hamburg, 2010.

\bibitem{strominger1986superstrings}
Andrew Strominger.
\newblock Superstrings with torsion.
\newblock {\em Nuclear Physics B}, 274(2):253--284, 1986.

\bibitem{strominger1996mirror}
Andrew Strominger, Shing-Tung Yau, and Eric Zaslow.
\newblock Mirror symmetry is t-duality.
\newblock {\em Nuclear Physics B}, 479(1-2):243--259, 1996.

\bibitem{tomassini2008symplectic}
Adriano Tomassini and Luigi Vezzoni.
\newblock On symplectic half-flat manifolds.
\newblock {\em manuscripta mathematica}, 125(4):515--530, 2008.

\end{thebibliography}

\Addresses

\end{document}